\documentclass[12pt,a4paper]{amsart}

\usepackage[utf8]{inputenc}
\usepackage[T1]{fontenc}
\usepackage[english]{babel}
\usepackage{cite}

\usepackage{indentfirst}
\usepackage{amssymb}
\usepackage{amsfonts}
\usepackage{amsmath}
\usepackage{amsthm}
\usepackage[top=2.5cm, bottom=2.5cm, left=2.5cm, right=2.5cm]{geometry}
\usepackage{amsopn}
\usepackage[colorlinks]{hyperref}
\usepackage{float}
\usepackage{graphicx}
\usepackage{color}
\usepackage{xspace,colortbl}
\usepackage{multicol}
\usepackage{multirow}

\newtheorem{theorem}{Theorem}[section]
\newtheorem{corollary}[theorem]{Corollary}

\theoremstyle{definition}

\newtheorem{remark}[theorem]{Remark}

\DeclareSymbolFont{bbsymbol}{U}{bbold}{m}{n}
\DeclareMathSymbol{\ind}{\mathbin}{bbsymbol}{'061}

\title[Nonparametric estimators of~inequality curves and inequality measures]{Nonparametric estimators of~inequality curves and inequality measures}
\author[A{.} Jokiel-Rokita]{Alicja Jokiel-Rokita}
\author[S{.} Pi\c{a}tek]{Sylwester Pi\c{a}tek}


\address[A.J.-R. and S.P.]{Faculty of Pure and Applied Mathematics\\ Wroc{\l}aw University
	of Science and Technology\\
	Wybrze\.ze Wyspia\'nskiego 27,
	50-370 Wroc{\l}aw, Poland
}
\email{alicja.jokiel-rokita@pwr.edu.pl}

\email{sylwester.piatek@pwr.edu.pl}

\begin{document}

\begin{abstract}
Classical inequality curves and inequality measures are defined for distributions with finite mean value.
Moreover, their empirical counterparts are not resistant to outliers.
For these reasons, quantile versions of known inequality curves such as the Lorenz, Bonferroni, Zenga and $D$ curves,
and quantile versions of inequality measures such as the Gini, Bonferroni, Zenga and $D$ indices have been proposed in the literature.
We propose various nonparametric estimators of quantile versions of inequality curves and inequality measures, prove their consistency,
and compare their accuracy in a~simulation study.
We also give examples of the use of quantile versions of inequality measures in real data analysis.
\end{abstract}

\maketitle

\section{Introduction}\label{sec:intro}

Classical inequality curves such as the Lorenz \cite{Lorenz1905}, Bonferroni~\cite{Bonferroni1930} and Zenga \cite{Zenga1984}, \cite{Zenga2007} curves,
as well as a~new inequality curve -- the $D$ curve, introduced in~\cite{Davydov2020}, are defined for distributions with a~finite mean value.
Therefore, the inequality measures related to the above-mentioned inequality curves also apply when the expected value of an observable random variable is finite.
In practice, we sometimes need to work with data from heavy-tailed distributions.
In such cases there is a~huge risk that the outliers might derange the analysis or that the expected value of the distribution may be infinite.
Such distributions appear in many situations where inequality measures (mostly the Gini index) are used, for example, in the analysis of income and wealth distribution \cite{Taleb2017} or financial returns and foreign exchange rates \cite{Ibragimov2017}.
For these reasons, quantile versions of known inequality curves such as the Lorenz, Bonferroni, Zenga and $D$ curves,
and quantile versions of inequality measures such as the Gini, Bonferroni, Zenga and $D$ indices have been proposed in the literature (see \cite{Prendergast2016} or \cite{Prendergast2018}).

The paper is organised as follows.
In Section~\ref{sec:curves} we recall the inequality curves, including the Lorenz, Bonferroni, Zenga, and $D$ curves.
The measures of inequality based on them are also given.
In Section~\ref{sec:QvIC} we describe quantile versions of the above-mentioned curves and indices.
Section \ref{sec:Estimation} is devoted to the nonparametric estimation of the curves and indices described in Section \ref{sec:QvIC}.
We propose plug-in estimators based on various estimators of the quantile function and prove their consistency.
In Section \ref{sec:simulation} we present the results of the simulation study which was conducted to compare the accuracy of the estimators of the inequality measures considered.
To illustrate applications of the quantile versions of the proposed inequality measures, in Section \ref{sec:real} two examples of real data analysis are provided.
Concluding remarks and prospect are given in Section \ref{sec:concluding}.

\section{Inequality curves and inequality measures}\label{sec:curves}

Let $X$ be a~non-negative-valued random variable with cumulative distribution function (cdf) $F(t)=P(X\leq t).$
Denote
\begin{equation*}
Q(p)=F^{-1}(p)=\inf\{t\in\mathbb{R}:F(t)\geq p\}
\end{equation*}
the quantile function (qf).

If $\mu=E(X)<\infty,$ then the Lorenz curve corresponding to the cdf $F$ and the qf $Q$ is defined by
\begin{equation}\label{e:Lcurve}
L(p;Q)=\frac{1}{\mu}\int_{0}^{p}Q(u)du, \ \ p\in(0,1).
\end{equation}
Formula (\ref{e:Lcurve}) is the Gastwirth \cite{Gastwirth1971} definition of the curve originally introduced by Lorenz \cite{Lorenz1905}.
For each $p\in(0,1),$ the value of $L(p;Q)$ expresses the share of income possessed by the $p$ percent of the poorer part of the population.
In the case of the Lorenz curve, the line of perfect equality is given by $y=x.$

The Bonferroni curve \cite{Bonferroni1930} can be defined by its relationship to the Lorenz curve. Namely,
\begin{equation}\label{e:Bcurve}
B(p;Q)=\frac{L(p;Q)}{p}, \ \ p\in(0,1).
\end{equation}
For each $p\in(0,1),$ the value of $B(p;Q)$ expresses the ratio of the mean income in the group of $p$ percent poorer part of the population
to the mean income in the population.
In the case of the Bonferroni curve, the line of perfect equality is given by $y=1.$

The Lorenz curve and the Bonferroni curve determine the parent distribution up to the scale factor.
In 1984 Zenga \cite{Zenga1984} introduced the inequality curve, called in the literature the Zenga-84 curve, which is problematic.
It is scale invariant, but it is possible that different distributions can have the same curve (see \cite{Arnold2015}).
In 2007 Zenga \cite{Zenga2007} introduced an alternative inequality curve defined by
\begin{equation}\label{e:Zcurve}
Z(p;Q)=1-\left[\frac{L(p;Q)}{p}\right] \left[\frac{1-p}{1-L(p;Q)}\right], \ \ p\in(0,1).
\end{equation}
In the literature, the Zenga inequality curve, given by (\ref{e:Zcurve}), is called the Zenga-07 curve.
Like the Lorenz and Bonferroni curves, it determines the parent distribution up to the scale factor.
In the case of the Zenga-07 curve, the line of perfect equality is given by $y=0.$
Denote
\begin{equation*}\label{e:Mcurve}
M(p;Q)=1-L(1-p;Q), \ \ p\in(0,1).
\end{equation*}
The $M(p;Q)$ curve provides the share of income owned by the richer $p$ percent of the population.
Using the notation of $M(p;Q)$, the Zenga-07 curve can be written by
\begin{equation*}\label{e:zc2}
Z(p;Q)=\frac{\frac{M(1-p;Q)}{1-p}-\frac{L(p;Q)}{p}}{\frac{M(1-p;Q)}{1-p}}, \ \ p\in(0,1).
\end{equation*}
For each $p\in(0,1),$ the value $Z(p;Q)$ is one minus the ratio of the mean income of the $p$ percent of the poorest to
the mean income of the remaining $1-p$ percent of the richest in a~population.
Motivated by the observed shifts towards extreme values in income distributions, Davydov and Greselin \cite{Davydov2020}
proposed a~new inequality curve given by
\begin{equation}\label{e:Dcurve}
D(p;Q)=1-\left[\frac{L(p;Q)}{p}\right]\left[\frac{p}{M(p;Q)}\right]=1-\frac{L(p;Q)}{M(p;Q)}, \ \ p\in(0,1).
\end{equation}
For each $p\in(0,1),$ the value $D(p;Q)$ is one minus the ratio of the mean income of the $p$ percent of the poorest to
the mean income of the $p$ percent of the richest in a~population.
In the case of the $D$ curve, the line of perfect equality is given by $y=0.$

Summary inequality measures have been defined for each of these inequality curves.
The classical measure of inequality associated with the Lorenz curve is called the Gini index and it is defined by
\begin{equation}\label{e:G}
GI(Q)=2\int_{0}^{1}[p-L(p;Q)]dp.
\end{equation}
Several alternative expressions for $GI(Q)$ are given in the literature (see e.g. \cite{Arnold1987}).
The Bonferroni $BI(Q)$, Zenga-07 $ZI(Q)$, and the $DI(Q)$ indices are defined similarly.
Namely,
\begin{equation}\label{e:BI}
BI(Q)=1-\int_{0}^{1}B(p;Q)\,dp,
\end{equation}
\begin{equation}\label{e:ZI}
ZI(Q)=\int_{0}^{1}Z(p;Q)\,dp,
\end{equation}
\begin{equation}\label{e:DI}
DI(Q)=\int_{0}^{1}D(p;Q)\,dp.
\end{equation}

Two families of inequality indices are considered in \cite{Eliazar2020}.
A $J$ curve, where $J(p;Q)=1-D(p;Q)$, was investigated together with its modification $J_m$ and an index $AJ_m$ by Gastwirth \cite{Gastwirth2016}.
They are based on the Lorenz curve, and we will not consider here their quantile versions.
Many other measures based on the Lorenz curve were discussed in the literature, see for example \cite{Sordo2013}, \cite{Gastwirth2014}, or \cite{Mukhopadhyay2021}.
Recently, a~generalised income inequality index was introduced in \cite{Dong2023}.
It is worth noting that the $ZI$ and $DI$ indices are not special cases of this index.

\section{Quantile version of inequality curves and inequality measures}\label{sec:QvIC}

Prendergast and Staudte \cite{Prendergast2016} introduced three quantile versions of the Lorenz curve.
Namely,
\begin{equation*}\label{e:L1}
L_{1}(p;Q)=p\frac{Q(p/2)}{Q(1/2)},
\end{equation*}
\begin{equation*}\label{e:L2}
L_{2}(p;Q)=p\frac{Q(p/2)}{Q(1-p/2)}
\end{equation*}
and
\begin{equation*}\label{e:L3}
L_{3}(p;Q)=p\frac{2Q(p/2)}{Q(p/2)+Q(1-p/2)}.
\end{equation*}
They also defined the following quantile versions of inequality measures
\begin{equation*}\label{e:Gi}
G_{i}(Q)=2\int_{0}^{1}[p-L_{i}(p;Q)]dp,
\end{equation*}
$i=1,2,3,$ analogously to the Gini index, given by (\ref{e:G}).

Taking into account the relation (\ref{e:Bcurve}) between the Bonferroni curve and the Lorenz curve, the
relation (\ref{e:Zcurve}) between the Zenga-07 curve and the Lorenz curve,
or the relation (\ref{e:Dcurve}) between the $D$ curve and the Lorenz curve,
we can also consider three quantile versions for each of them,
and three new quantile versions of $BI,$ $ZI$ and $DI$.
However, on the basis of interpretation of the curves $B(p;Q),$ $Z(p;Q)$ and $D(p;Q)$ for each $p$,
we define the following quantile versions
\begin{equation}\label{e:qB}
qB(p;Q)=\frac{Q(p/2)}{Q(1/2)}=\frac{L_{1}}{p},
\end{equation}
\begin{equation}\label{e:qZ}
qZ(p;Q)=1-\frac{Q(p/2)}{Q((1+p)/2)}
\end{equation}
and
\begin{equation}\label{e:qD}
qD(p;Q)=1-\frac{Q(p/2)}{Q(1-p/2)}=1-\frac{L_{2}}{p}
\end{equation}
of the Bonferroni, Zenga-07 and $D$ curve, respectively.
In \cite{Prendergast2018} an inequality curve is considered which can be regarded as a~quantile version of the $D$ curve.
Namely, the symmetric ratio of quantiles $R(p)=Q(p/2)/Q(1-p/2)=L_{2}/p$ is introduced. The curve $qD(p;Q)=1-R(p).$
Thus, in a~sense, the quantile version of the $D$ curve was introduced earlier than the $D$ curve.

Let us notice that the curves $qB,$ $qZ,$ $qD$ are scale invariant,
but it is possible that different distributions can have the same quantile version of the Bonferroni curve $qB$.
For this reason, in the sequel we will only consider nonparametric estimation of the curves $qZ,$ $qD,$ and the following
quantile versions
\begin{equation*}\label{e:qZI}
qZI(Q)=\int_{0}^{1}qZ(p;Q)\,dp,
\end{equation*}
\begin{equation*}\label{e:qDI}
qDI(Q)=\int_{0}^{1}qD(p;Q)\,dp=1-\int_{0}^{1}R(p)\,dp
\end{equation*}
of the inequality measures $ZI,$ $DI$, respectively. The index $qZI(Q)$ was mentioned in \cite{Prendergast2018} as an alternative to the Zenga index.

For each $p\in(0,1),$ the value $qZ(p;Q)$ is one minus the ratio of the median income of the $p$ percent of the poorest to
the median income of the remaining $1-p$ percent of the richest in~population.
In the case of $qD,$ for each $p\in(0,1),$ the value $qD(p;Q)$ is one minus the ratio of the median income of the $p$ percent of the poorest to
the median income of the $p$ percent of the richest in a~population.
For both the $qZ$ and $qD$ curve, the line of perfect equality is given by $y=0.$
Figures \ref{fig:qZ_example_5} -- \ref{fig:qD_example_10} contain examples of the $qZ$ and $qD$ curves plots, corresponding to the Dagum distribution
$\mathcal{D}(\sigma ,a,b)$ with the cdf
\begin{equation}\label{e:FDagum}
F(x)=\Big[1+\Big(\frac{x}{\sigma}\Big)^{-a}\Big]^{-b},
\end{equation}
for $x>0.$

\par\noindent
\begin{minipage}[b]{0.45\textwidth}
\begin{figure}[H]
\includegraphics[angle=0,width=\textwidth]
{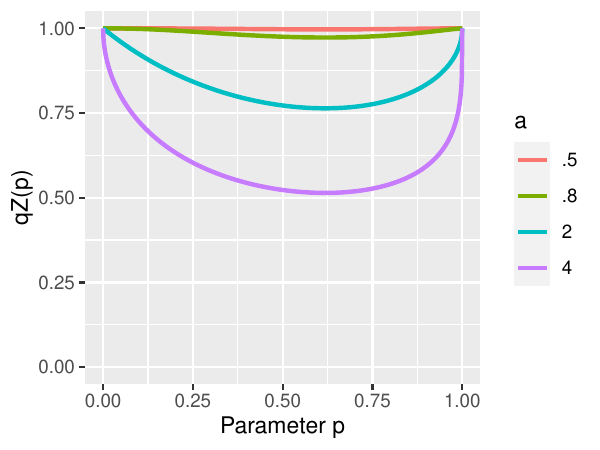}
\caption[Plots of $qZ$ curves corresponding to the Dagum distribution with parameter $b=.5$]{Plots of $qZ$ curves corresponding to the Dagum distribution with parameter $b=.5$}
\label{fig:qZ_example_5}
\end{figure}
\end{minipage}
\hfill
\begin{minipage}[b]{0.45\textwidth}
\begin{figure}[H]
\includegraphics[angle=0,width=\textwidth]
{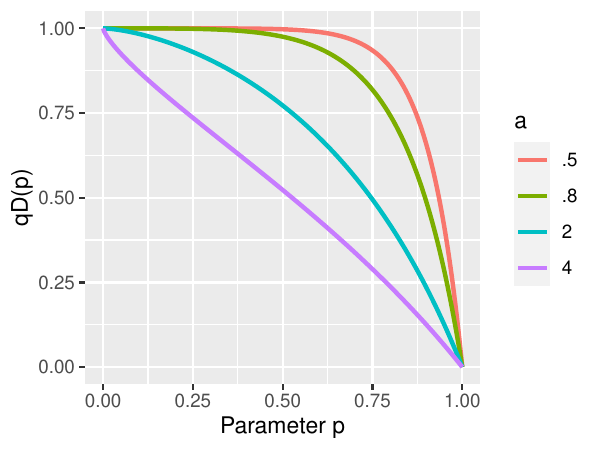}
\caption[Plots of $qD$ curves corresponding to the Dagum distribution with parameter $b=.5$]{Plots of $qD$ curves corresponding to the Dagum distribution with parameter $b=.5$}
\label{fig:qD_example_5}
\end{figure}
\end{minipage}
\par\vspace{1ex}

\par\noindent
\begin{minipage}[b]{0.45\textwidth}
\begin{figure}[H]
\includegraphics[angle=0,width=\textwidth]
{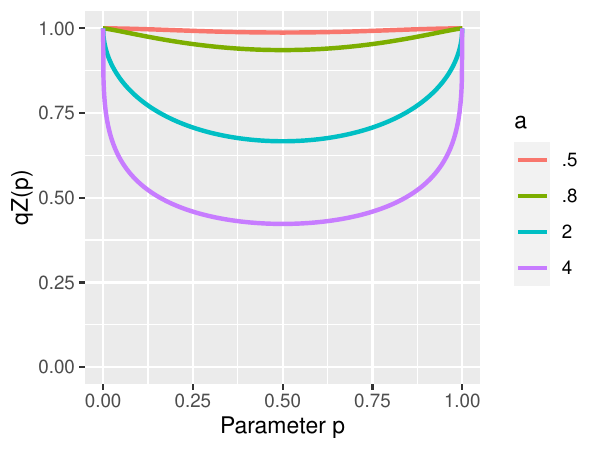}
\caption[Plots of $qZ$ curves corresponding to the Dagum distribution with parameter $b=1$]{Plots of $qZ$ curves corresponding to the Dagum distribution with parameter $b=1$}
\label{fig:qZ_example_10}
\end{figure}
\end{minipage}
\hfill
\begin{minipage}[b]{0.45\textwidth}
\begin{figure}[H]
\includegraphics[angle=0,width=\textwidth]
{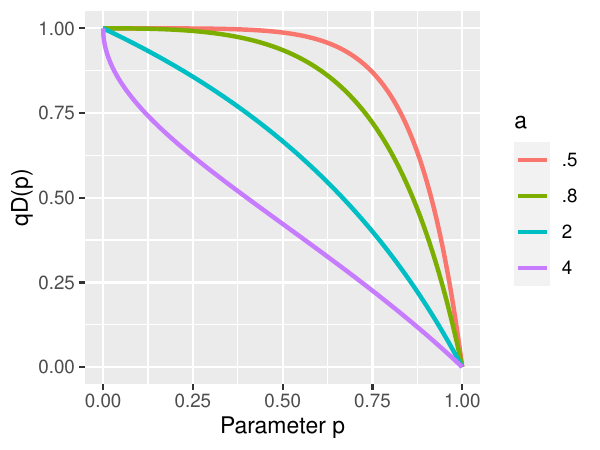}
\caption[Plots of $qD$ curves corresponding to the Dagum distribution with parameter $b=1$]{Plots of $qD$ curves corresponding to the Dagum distribution with parameter $b=1$}
\label{fig:qD_example_10}
\end{figure}
\end{minipage}
\par\vspace{1ex}

In Figures \ref{fig:qZ_example_5} and \ref{fig:qZ_example_10} we can notice that for $a=0.5$ and $a=0.8$, the $qZ$ curve is very close to 1.
Therefore, the $qD$ curve, presented in Figures \ref{fig:qD_example_5} and \ref{fig:qD_example_10} , may be more useful in comparing distributions with infinite expected value.

The interpretation of the measures $qZI$ and $qDI$ is as follows.
Let $F$ be a~continuous cdf on $[0,\infty]$ with median $m=Q(1/2),$
and let $X^{*}$ be a~randomly chosen income from those incomes less than the median $m.$
Then, $X^{*}\sim F^{*},$ where $F^{*}(x)=2F(x)$ for $x<m.$
Next, let $Y^{*}=Q(1/2+F(X^{*})),$ so if $X^{*}=Q(r),$ then $Y^{*}=Q(1/2+r).$
It can be shown that
\begin{equation}\label{eq:bs_qzi}
qZI(Q)=E_{Q}\left(\frac{Y^{*}-X^{*}}{Y^{*}}\right).
\end{equation}
Let $Y^{**}=Q(1-F(X^{*})).$ Then, if $X^{*}=Q(r),$ then $Y^{**}=Q(1-r).$
It can be shown that
\begin{equation}\label{eq:bs_qdi}
qDI(Q)=E_{Q}\left(\frac{Y^{**}-X^{*}}{Y^{**}}\right).
\end{equation}

It can be easily shown that the indices $qZI(Q)$ and $qDI(Q)$ satisfy two basic properties: scale invariance and decrease under uniform increase in incomes.
They can also be used with the zero-inflated distributions.

\section{Nonparametric estimators of the inequality curves and the inequality measures}\label{sec:Estimation}

Given a~random sample ${\bf X}=(X_{1},\ldots ,X_{n}),$ from an unknown continuous distribution function $F,$ with the support $(0,\infty),$ and quantile function $Q$,
we are interested in estimation of the curves $qZ(p;Q),$ $qD(p;Q),$ and indices $qZI(Q),$ $qDI(Q).$
We consider the plug-in estimators
\begin{equation}\label{e:qZest}
\widehat{qZ}(p;{\bf X})=qZ(p;{\widehat Q}(p;{\bf X})),
\end{equation}
\begin{equation}\label{e:qDest}
\widehat{qD}(p;{\bf X})=qD(p;{\widehat Q}(p;{\bf X}))
\end{equation}
of the curves $qZ(p;Q),$ $qD(p;Q),$ respectively, and
\begin{equation}\label{e:qZIest}
\widehat{qZI}({\bf X})=\int_{0}^{1}qZ(p;{\widehat Q}(p;{\bf X}))\,dp,
\end{equation}
\begin{equation}\label{e:qDIest}
\widehat{qDI}({\bf X})=\int_{0}^{1}qD(p;{\widehat Q}(p;{\bf X}))\,dp
\end{equation}
of the indices $qZI(Q),$ $qDI(Q),$
where ${\widehat Q}(p;{\bf X})$ is an estimator of the quantile function $Q(p).$
A natural estimator of the $Q(p)$ is the empirical quantile function
\begin{equation}\label{e:Qn}
\widehat{Q}_{n}^{E}(p;{\bf X}):=\inf\{t:{\widehat F}_{n}^{E}(t;{\bf X})\geq p\},
\end{equation}
where
\begin{equation}\label{EDF}
{\widehat F}_{n}^{E}(t;{\bf X}):=\frac{1}{n}\sum_{i=1}^{n}I_{(-\infty,t]}(X_{i})
\end{equation}
is the empirical distribution function, and $I_{A}(x)=1$ if $x\in A$ and $0$~otherwise.

For each $p\in(0,1),$ the estimator ${\widehat Q}_{n}^{E}(p;{\bf X}),$ given by (\ref{e:Qn}), is called the sample quantile of order $p.$
In the literature and statistical software there is a~large number of different definitions used for sample quantiles.
In a~widely cited article \cite{Hyndman1996}, the authors analysed nine different sample quantile definitions.
Most of them are based on quantile function estimators ${\widehat Q}_{n}$ constructed by linear interpolation between so-called plotting positions,
i.e., the points $p_{k},$ $k=1,\ldots,n,$ for which ${\widehat Q}(p_{k})=X_{k:n},$
where $X_{1:n},\ldots ,X_{n:n}$ denote the order statistics of the sample ${\bf X}_{n}.$
In \cite{Hyndman1996} the authors have compared sample quantiles describing their motivation and whether or not they possess some desired properties.
Among the estimators compared in \cite{Hyndman1996}, only the estimator proposed by Hazen in \cite{Hazen1914}, which is based on the plotting positions
\begin{equation}\label{ppH}
p_{k}^{H}=\frac{k-1/2}{n}, \qquad k=1,\ldots,n,
\end{equation}satisfies all six properties.
However, Hyndman and Fan in \cite{Hyndman1996} recommend the estimator, which is based on the plotting positions
\begin{equation}\label{ppHF}
p_{k}^{HF}=\frac{k-1/3}{n+1/3}, \qquad k=1,\ldots,n,
\end{equation}
because it gives approximately median-unbiased estimates of $Q(p)$ regardless of the distribution.
The problem of a~need to adopt a~standard definition for sample quantiles was also discussed by Langford \cite{Langford2006},
who identified twelve different sample quantile definitions that are used in statistical software.
Makkonen and Pajari in \cite{Makkonen2014} recommend the sample quantiles, proposed by Weibull \cite{Weibull1939} and Gumbel \cite{Gumbel1939}, based on plotting positions
\begin{equation}\label{ppWG}
p_{k}^{WG}=\frac{k}{n+1}, \qquad k=1,\ldots,n.
\end{equation}

In the sequel, the estimators of the quantile function $Q(p)$ based on plotting position $p_{k}^{H},$ given by (\ref{ppH}), $p_{k}^{HF},$ given by (\ref{ppHF}),
$p_{k}^{WG},$ given by (\ref{ppWG}), and the sample ${\bf X}=(X_{1},\ldots,X_{n})$ of size $n$, will be denoted by ${\widehat Q}_{n}^{H}(p;{\bf X}),$
${\widehat Q}_{n}^{HF}(p;{\bf X}),$ ${\widehat Q}_{n}^{WG}(p;{\bf X}),$ respectively.
Analogously, we will denote the plug-in estimators of the curves $qZ,$ $qD$ and the indices $qZI,$ $qDI$.
For example, the plug-in estimator of $qZ$ and $qZI$ based on ${\widehat Q}_{n}^{H}(p;{\bf X})$ will be denoted by
$\widehat{qZ}_{n}^{H}(p;{\bf X})$ and $\widehat{qZI}_{n}^{H}({\bf X}),$ respectively.

\begin{remark}
    Since the estimators $\widehat{qZ}_{n}^{E}, \widehat{qZ}_{n}^{H}, \widehat{qZ}_{n}^{WG}$ and $\widehat{qZ}_{n}^{HF}$ are defined as a~ratio of two linear functions of $p$, they can be easily integrated analytically. This leads to the closed-form expression for $\widehat{qZ}_{n}^{E}, \widehat{qZI}_{n}^{H}, \widehat{qZI}_{n}^{WG}$ and $\widehat{qZI}_{n}^{HF}$, depending on the values of $p_k$ and $X_{k:n}$. The analogous property holds also for estimators of $qD$ and $qDI$.
\end{remark}

\section{Properties of the plug-in estimators of the quantile versions of the inequality curves and the inequality measures}

In this section, we prove the strong consistency of the plug-in estimators of $qZ$, $qD$, $qZI$ and $qDI$.
We also show that the appropriately scaled empirical processes of $qZ$ and $qD$ converge to a~Gaussian process.

Let $\widetilde{F}_{n}(t;{\bf X})$ be an estimator of cdf $F$ such that
\begin{equation*}\label{e:Ftilde}
\sup_{t\in{\mathbb R}}|\widetilde{F}_{n}(t;{\bf X})-\widehat{F}_{n}^{E}(t;{\bf X})|\leq \frac{1}{n} \ \ {\rm a.s.}
\end{equation*}
Denote
\begin{equation}\label{e:Qtilde}
\widetilde{Q}_{n}(p;{\bf X})=\inf\{t:\widetilde{F}_{n}(t;{\bf X})\geq p\},
\end{equation}
and $\widetilde{qZ}_{n}(p;{\bf X}),$ $\widetilde{qD}_{n}(p;{\bf X})$ the plug-in estimator, based on the estimator $\widetilde{Q}_{n}(p;{\bf X}),$
of $qZ(p;Q),$ $qD(p;Q),$ respectively.

\begin{remark}\label{rem:estTilde}
From formulas (\ref{ppH}), (\ref{ppHF}) and (\ref{ppWG}), it can be easily shown that the estimators
${\widehat Q}_{n}^{H}(p;{\bf X}),$ ${\widehat Q}_{n}^{HF}(p;{\bf X}),$ ${\widehat Q}_{n}^{WG}(p;{\bf X})$ are special cases of $\widetilde{Q}_{n}(p;{\bf X}),$
given by (\ref{e:Qtilde}). The estimator ${\widehat Q}_{n}^{E}(p;{\bf X})$ is, of course, also a~special case of $\widetilde{Q}_{n}(p;{\bf X}).$
\end{remark}

\begin{theorem}\label{th:pointConv}
Let ${\bf X}=(X_{1},\ldots,X_{n})$ be a~sample from the distribution with a~continuous cdf $F,$ and the support $(0,\infty).$
Then, for each $p\in(0,1),$
\begin{equation*}\label{e:pointConvZ}
\widetilde{qZ}_{n}(p;{\bf X})\xrightarrow{a.s.} qZ(p;Q), \ \  when \ \ n\rightarrow\infty ,
\end{equation*}
\begin{equation*}\label{e:pointConvD}
\widetilde{qD}_{n}(p;{\bf X})\xrightarrow{a.s.} qD(p;Q), \ \  when \ \ n\rightarrow\infty .
\end{equation*}
\end{theorem}
\begin{proof}
If ${\bf X}$ satisfies the assumptions of the theorem, then from Theorem 3.2 in~\cite{JokSied2019}, we have that for each $p\in(0,1),$
\begin{equation*}
\widetilde{Q}_{n}(p;{\bf X})\xrightarrow{a.s.} Q(p)\ \  when \ \ n\rightarrow\infty ,
\end{equation*}
and $P_{F}(\widetilde{Q}_{n}(p;{\bf X})\neq 0)=1.$
Therefore, the theorem follows from the definition of the almost sure convergence and formulas (\ref{e:qZ}), (\ref{e:qD}).
\end{proof}

\begin{theorem}\label{th:supConv}
Under the assumptions of Theorem \ref{th:pointConv},
\begin{equation*}\label{e:supConvZ}
\sup_{p\in(0,1)}|\widetilde{qZ}_{n}(p;{\bf X})-qZ(p;Q)| \xrightarrow{a.s.} 0, \ \  when \ \ n\rightarrow\infty ,
\end{equation*}
\begin{equation*}\label{e:supConvD}
\sup_{p\in(0,1)}|\widetilde{qD}_{n}(p;{\bf X})-qD(p;Q)| \xrightarrow{a.s.} 0, \ \  when \ \ n\rightarrow\infty .
\end{equation*}
\end{theorem}
\begin{proof}
Based on Theorem \ref{th:pointConv}, for each realisation ${\bf x}$ of the sample ${\bf X},$ the sequence of functions $\widetilde{qZ}_{n}(p;{\bf x})$
of the argument $p,$ converges pointwise to $qZ(p;Q).$
Furthermore, under the assumptions concerning the cdf $F,$ the $qZ(p;Q),$ as the function of~$p,$ is uniformly continuous on $[0,1].$
Therefore, for each realisation ${\bf x}$ of the sample ${\bf X},$ the sequence of functions $\widetilde{qZ}_{n}(p;{\bf x})$
converges to $qZ(p;Q)$ uniformly.
The proof of the second part of the theorem is analogous and will be omitted.
\end{proof}

\begin{corollary}
Under the assumptions of Theorem \ref{th:pointConv}, taking into account Remark~\ref{rem:estTilde},
\begin{enumerate}
\item[(i)]
$\widehat{qZ}_{n}^{E}(p;{\bf X}),$  $\widehat{qZ}_{n}^{H}(p;{\bf X}),$
$\widehat{qZ}_{n}^{HF}(p;{\bf X}),$ $\widehat{qZ}_{n}^{WG}(p;{\bf X})$ are strongly consistent estimators of $qZ(p;Q),$
\item[(ii)]
$\widehat{qD}_{n}^{E}(p;{\bf X}),$ $\widehat{qD}_{n}^{H}(p;{\bf X}),$
$\widehat{qD}_{n}^{HF}(p;{\bf X}),$ $\widehat{qD}_{n}^{WG}(p;{\bf X})$ are strongly consistent estimators of $qD(p;Q).$
\end{enumerate}
\end{corollary}

\begin{theorem}\label{th:IndexConv}
Under the assumptions of Theorem \ref{th:pointConv},
\begin{equation*}\label{e:pointConvZI}
\widetilde{qZI}_{n}({\bf X})\xrightarrow{a.s.} qZI(Q), \ \  when \ \ n\rightarrow\infty ,
\end{equation*}
\begin{equation*}\label{e:pointConvDI}
\widetilde{qDI}_{n}({\bf X})\xrightarrow{a.s.} qDI(Q), \ \  when \ \ n\rightarrow\infty .
\end{equation*}
\end{theorem}
\begin{proof}
From the definition of $\widetilde{qZI}_{n}({\bf X})$ and $qZI(Q),$ we have
\begin{equation*}
\widetilde{qZI}_{n}({\bf X})-qZI(Q)=\int_{0}^{1}
\left[\widetilde{qZ}_{n}(p;{\bf X})-qZ(p;Q)\right]\,dp.
\end{equation*}
The theorem follows from the inequality
\begin{eqnarray*}
\left|\int_{0}^{1}\left[\widetilde{qZ}_{n}(p;{\bf X})-qZ(p;Q)\right]\,dp \right|&\leq & \sup_{p\in(0,1)}|\widetilde{qZ}_{n}(p;{\bf X})-qZ(p;Q)|,
\end{eqnarray*}
and Theorem \ref{th:supConv}.
The proof of the second part of the theorem is analogous and will be omitted.
\end{proof}

\begin{corollary}
Under the assumptions of Theorem \ref{th:pointConv}, taking into account Remark~\ref{rem:estTilde},
\begin{enumerate}
\item[(i)]
$\widehat{qZI}_{n}^{E}({\bf X}),$ $\widehat{qZI}_{n}^{H}({\bf X}),$ $\widehat{qZI}_{n}^{HF}({\bf X}),$ $\widehat{qZI}_{n}^{WG}({\bf X})$ are strongly consistent estimators of $qZI(Q),$
\item[(ii)]
$\widehat{qDI}_{n}^{E}({\bf X}),$ $\widehat{qDI}_{n}^{H}({\bf X}),$ $\widehat{qDI}_{n}^{HF}({\bf X}),$ $\widehat{qDI}_{n}^{WG}({\bf X})$ are strongly consistent estimators of  $qDI(Q).$
\end{enumerate}
\end{corollary}

Denote
\begin{align}
    \xi_n^Z(p)=\sqrt{n}\left[\widehat{qZ}^E_n(p)-qZ(p)\right],
    \quad p\in[0,1],
\end{align}
and
\begin{align}
    \xi_n^D(p)=\sqrt{n}\left[\widehat{qD}^E_n(p)-qD(p)\right],
    \quad p\in[0,1]
\end{align}
random elements with values in a metric space of all bounded functions $z:[0,1]\rightarrow{\mathbb R}$ with the norm $\|z\|=\sup_{p\in[0,1]}|z(p)|.$
\begin{theorem}\label{th:BB}
    Under the assumptions of Theorem \ref{th:pointConv}, and assuming that $Q^{'}(p)=1/(f(Q(p))$ is bounded on any interval $(a,b),$ $0<a<b<1,$
    \begin{enumerate}
    \item[(i)]
      the process $\xi_n^Z(p)$ converges in probability on an interval $[a,b]$  to the Gaussian process
\begin{align*}
    G^Z(p)=\left[1-qZ(p)\right]
    \left[ \frac{Q^{'}\left(\frac{p}{2}\right)}{Q\left(\frac{p}{2}\right)}B\left(\frac{p}{2}\right)
    -\frac{Q^{'}\left(\frac{1+p}{2}\right)}{Q\left(\frac{1+p}{2}\right)}B\left(\frac{1+p}{2}\right)\right],
\end{align*}
\item[(ii)]
the process $\xi_n^D(p)$ converges in probability on an interval $[a,b]$ to the Gaussian process
\begin{align*}
    G^D(p)=\left[1-qD(p)\right]
    \left[ \frac{Q^{'}\left(\frac{p}{2}\right)}{Q\left(\frac{p}{2}\right)}B\left(\frac{p}{2}\right)
    -\frac{Q^{'}\left(1-\frac{p}{2}\right)}{Q\left(1-\frac{p}{2}\right)}B\left(1-\frac{p}{2}\right)\right],
\end{align*}
\end{enumerate}
where $B(p),$ $p\in(0,1),$ is the standard Brownian bridge.
\end{theorem}
\begin{proof}
(i) Throughtout this proof we will denote $Q_n(p):=\widehat{Q}^{E}_n(p)$ and $A_n(p)=\sqrt{n}(Q_{n}(p)-Q(p))$.
We transform $\xi_n^Z(p)$ as follows:
\begin{align}
\label{eq:product_2_rv}
\nonumber
\xi_n^Z(p)&=
     \sqrt{n}\Bigg[ \frac{Q_{n}\left(\frac{p}{2}\right)}{Q_{n}\left(\frac{1+p}{2}\right)} -  \frac{Q\left(\frac{p}{2}\right)}{Q_n(\frac{1+p}{2})}
    +\frac{Q\left(\frac{p}{2}\right)}{Q_n\left(\frac{1+p}{2}\right)} -\frac{Q(\frac{p}{2})}{Q(\frac{1+p}{2})} \Bigg]\\ \nonumber
    &= \frac{1}{Q_n\left(\frac{1+p}{2}\right)}
    \Bigg\{ A_n\left(\frac{p}{2}\right)-\sqrt{n}\left[
    \frac{Q\left(\frac{p}{2}\right)}{Q\left(\frac{1+p}{2}\right)}Q_n\left(\frac{1+p}{2}\right)-Q\left(\frac{p}{2}\right)
    \right]
    \Bigg\}\\
    &= \frac{1}{Q_n\left(\frac{1+p}{2}\right)}
    \left[
    A_n\left(\frac{p}{2}\right)-\left[1-qZ(p)\right]A_n\left(\frac{1+p}{2}\right)
    \right],
\end{align}
which is a~product of two random elements.
Since random elements $Q_n\xrightarrow{d}Q$ (see for example Lemma 21.2 in \cite{VanDerVaart1998}), and $Q$ is nonrandom, we have $Q_n\xrightarrow{P}Q.$
Basing on this and using the continuous mapping theorem (see for example Theorem 2.3 in \cite{VanDerVaart1998}), we get
    \begin{align}\label{eq:quantile_conv}
    \frac{1}{Q_n\left(\frac{1+p}{2}\right)} \xrightarrow{P}
    \frac{1}{Q\left(\frac{1+p}{2}\right)}.
\end{align}
We also know (see for example Theorem 18.1.1, p. 640, in \cite{Shorack1989} or Theorem 3.1 in \cite{Bickel1967}) that if $Q^{'}$ is positive and continuous on an open subinterval of $[0, 1]$ containing an interval $[a, b]$, then
\begin{align}
    \sup_{a\leq p\leq b}|A_n(p)/Q^{'}(p)-B(p)|\xrightarrow{P}0,
\end{align}
which according to the definition of convergence in probability in metric spaces implies that
$A_n(p)\xrightarrow{P}Q^{'}(p)B(p)$.
Hence, we have (see for example Theorem 2.7 in \cite{VanDerVaart1998}) the convergence
\begin{align}\label{eq:vector_conv}
    &\left[
    A_n\left(\frac{p}{2}\right),\left[1-qZ(p)\right]A_n\left(\frac{1+p}{2}\right)
    \right]\\
    &\xrightarrow{P}
    \left[
Q^{'}\left(\frac{p}{2}\right)B\left(\frac{p}{2}\right)
,\left[1-qZ(p)\right]
Q^{'}\left(\frac{1+p}{2}\right)B\left(\frac{1+p}{2}\right)
    \right].
\end{align}
Thus, according to the continuous mapping theorem, the second factor in (\ref{eq:product_2_rv}) converges in probability to the difference of the random elements on the right-hand side of (\ref{eq:vector_conv}), and finally, the product given by (\ref{eq:product_2_rv}) converges in probability to
\begin{equation*}
\frac{1}{Q\left(\frac{1+p}{2}\right)}
    \left[ Q^{'}\left(\frac{p}{2}\right)B\left(\frac{p}{2}\right)
    -\left[1-qZ(p)\right]Q^{'}\left(\frac{1+p}{2}\right)B\left(\frac{1+p}{2}\right)\right]=G^{Z}(p).
\end{equation*}

(ii) The proof of part (ii) is analogous to (i) and will be omitted.
\end{proof}

\begin{remark}
The pointwise convergence of the ratio of two empirical quantiles to the Gaussian random variable was shown by Jędrzejczak et. al. (see Theorem 2.2 in \cite{Jedrzejczak2021}),
and their result can be obtained as a special case of Theorem \ref{th:BB}.
\end{remark}

\begin{theorem} The estimators $\widehat{qZI}^E_n$ and $\widehat{qDI}^E_n$ are asymptotically normal, namely:
    \begin{enumerate}
    \item[(i)]
    $\sqrt{n}\left(\widehat{qZI}^E_n-qZI\right)$ converges in probability to the normally distributed random variable with mean 0 and variance $\sigma^2_Z$,
    \item[(ii)]
    $\sqrt{n}\left(\widehat{qDI}^E_n-qDI\right)$ converges in probability to the normally distributed random variable with mean 0 and variance $\sigma^2_D$,
    \end{enumerate}
where $\sigma^2_Z$ and $\sigma^2_D$ are given in the proof.
\end{theorem}
\begin{proof}
    From the definition of $\widehat{qZI}^E_n$ and $qZI$
    \begin{align}
    \sqrt{n}\left[\widehat{qZI}^E_n-qZI\right] =
    \sqrt{n}\int_0^1\left(\widehat{qZ}^E_n(p)-qZ(p)\right)\,dp,
    \end{align}
    which converges in probability to $\int_0^1 G^Z(p)\,dp$, following Theorem \ref{th:BB} and the fact that $G^Z(p)$ is bounded and the integral is defined on a~closed interval.
    Since an integral of a~bounded Gaussian process is a~Gaussian random variable, we only need to compute its mean and variance. Mean is equal to zero, because
    \begin{align}
    \mathbb{E} \int_0^1 G^Z(p)\,dp = \int_0^1 \mathbb{E} G^Z(p)\,dp =0.
    \end{align}

    Denote
\begin{align*}
    a(p)=\left[1-qZ(p)\right]\frac{Q^{'}\left(\frac{p}{2}\right)}{Q\left(\frac{p}{2}\right)} \qquad \text{and} \qquad
    b(p)=\left[1-qZ(p)\right]\frac{Q^{'}\left(\frac{1+p}{2}\right)}{Q\left(\frac{1+p}{2}\right)}.
\end{align*}
Then 
\begin{align*}
    G^Z(p)&=
    a(p)B\left(\frac{p}{2}\right)
    b(p)\left(\frac{1+p}{2}\right)\\
    &=a(p)\biggl[W\left(\frac{p}{2}\right)-\frac{p}{2}W(1)\biggr]
    +b(p)\biggl[W\left(\frac{1+p}{2}\right)-\frac{1+p}{2}W(1)\biggr],
\end{align*}
where $W(t)$ is a standard Wiener process. From properties of the Wiener process, it can be shown that
\begin{align*}
\mathbb{E}\left[\int_0^1 G^Z(p)\,dp\right]^2=
&\int_0^1\int_0^1\biggl\{
a(p)a(q)\frac{1}{2}\min\{p,q\}
+b(p)b(q)\left[\frac{1}{2}+\frac{1}{2}\min\{p,q\}\right]\\
&+\left[b(p)\frac{1+p}{2}-a(p)\frac{p}{2}\right]
\left[b(q)\frac{1+q}{2}-a(q)\frac{q}{2}\right]\\
&-a(p)b(q)\frac{p}{2}-b(p)a(q)\frac{q}{2}\\
&+\left[a(p)\frac{p}{2}-b(p)\frac{1+p}{2}\right]\left[b(q)\frac{1+q}{2}-a(q)\frac{q}{2}\right]\\
&+\left[a(q)\frac{q}{2}-b(q)\frac{1+q}{2}\right]\left[b(p)\frac{1+p}{2}-
a(p)\frac{p}{2}\right]
\biggr\}dq\,dp,
\end{align*}
which reduces to
\begin{align*}
&\int_0^1\int_0^1
\biggl\{
a(p)a(q)\left[\frac{1}{2}\min\{p,q\}-\frac{pq}{4}\right]\\
&+b(p)b(q)\left[\frac{1}{2}+\frac{1}{2}\min\{p,q\}-\frac{\left(1+p\right)\left(1+q\right)}{4}\right]\\
&+b(p)a(q)\frac{q}{2}\frac{p-1}{2}+
a(p)b(q)\frac{p}{2}\frac{q-1}{2}
\biggr\}
dq\,dp=\sigma^2_Z.
\end{align*}
It can be shown analogously that
\begin{align*}
\sigma^2_D=&\int_0^1\int_0^1
\biggl\{
c(p)c(q)\left[\frac{1}{2}\min\{p,q\}-\frac{pq}{4}\right]\\
&+d(p)d(q)\left[1-\frac{1}{2}\max\{p,q\}-\left(1-\frac{p}{2}\right)\left(1-\frac{q}{2}\right)\right]\\
&-\frac{pq}{4}\bigl[d(p)c(q)+
c(p)d(q)\bigr]
\biggr\}
dq\,dp,
\end{align*}
where
\begin{align*}
    c(p)=\left[1-qD(p)\right]\frac{Q^{'}\left(\frac{p}{2}\right)}{Q\left(\frac{p}{2}\right)} \qquad \text{and} \qquad
    d(p)=\left[1-qD(p)\right]\frac{Q^{'}\left(1-\frac{p}{2}\right)}{Q\left(1-\frac{p}{2}\right)}.
\end{align*}

\end{proof}

The variances $\sigma^2_Z$ and $\sigma^2_D$ are close to 0 if the value of the index is close to 0 or 1. The precise relation between the variance and the value of the index for Dagum distribution $\mathcal{D}(1,a,1)$ with different values of shape parameter $a$ is shown in Figures \ref{fig:qZ_var} and \ref{fig:qD_var} for $qZI$ and $qDI$ respectively. The values of $qZI$ and $qDI$ in these plots depend on $a$ parameter.


\begin{minipage}[b]{0.45\textwidth}
\begin{figure}[H]
\includegraphics[angle=0,width=\textwidth]
{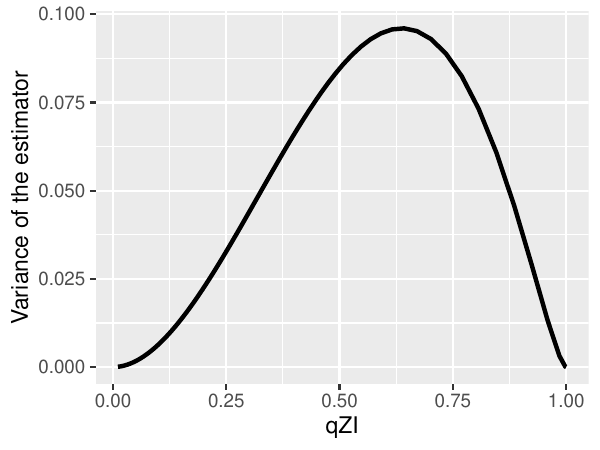}
\caption[$\sigma^2_Z$ of $\widehat{qZI}^E$ estimator for $\mathcal{D}(1,a,1)$]{$\sigma^2_Z$ of $\widehat{qZI}^E$ estimator for $\mathcal{D}(1,a,1)$}
\label{fig:qZ_var}
\end{figure}
\end{minipage}
\hfill
\begin{minipage}[b]{0.45\textwidth}
\begin{figure}[H]
\includegraphics[angle=0,width=\textwidth]
{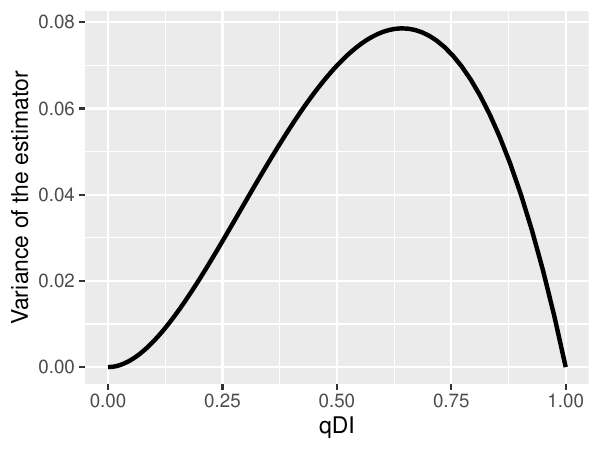}
\caption[$\sigma^2_D$ of $\widehat{qDI}^E$ estimator for $\mathcal{D}(1,a,1)$]{$\sigma^2_D$ of $\widehat{qDI}^E$ estimator for $\mathcal{D}(1,a,1)$}
\label{fig:qD_var}
\end{figure}
\end{minipage}

The asymptotic behaviour of the empirical estimator of the Lorenz curve was covered in the literature (see Section 3.5 in \cite{Goldie1977}). It was studied also for Gini index with some assumptions regarding the distribution of the observations (see for example Section 4~in \cite{Goldie1977}), for Chakravarty's \cite{Chakravarty1988} extended Gini index (see Theorem 3.2 in \cite{Zitikis2003}) and for quintile share ratio (QSR) (see Theorem 2~in \cite{Kpanzou2014}). The last result is more general and can be extended also to concentration measures defined in a~similar manner to QSR, such as decile share ratio (DSR) or Palma ratio \cite{Palma2011}.

\begin{table}[h]
\centering
\caption{Values of the $qZI$ and $qDI$ for Dagum distribution with different values of shape parameters $a$ and $b$}
\label{tab:index_values}
\begin{tabular}{c|cccc|cccc}
\hline
\multicolumn{1}{c|}{ $b$ } & \multicolumn{4}{c|}{$.5$} & \multicolumn{4}{c}{$1$} \\
\hline
$a$  & .5 & .8 & 2 & 4 & .5 & .8 & 2 & 4 \\
\hline
$qZI$  & 0.9985  & 0.9849 &  0.8288 &  0.5973 &  0.9932 &  0.9589 &  0.7344 &  0.4912 \\
$qDI$ &  0.9079  & 0.8563  & 0.6877 &  0.5105  & 0.8785 &  0.8127  & 0.6137  & 0.4292 \\
\hline
\end{tabular}
\end{table}


\begin{minipage}[b]{0.45\textwidth}
\begin{figure}[H]
\includegraphics[angle=0,width=\textwidth]
{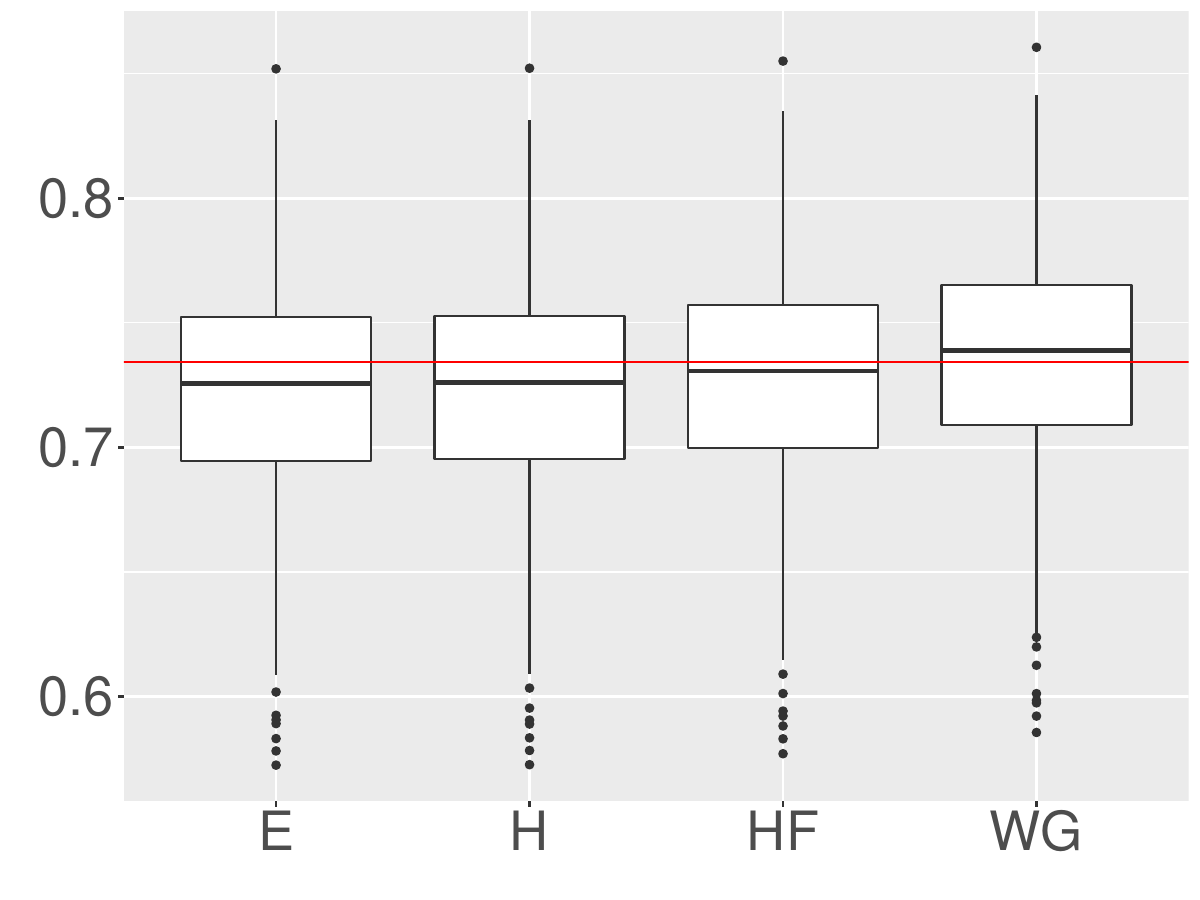}
\caption[Boxplots of estimated values of $qZI$ for the $\mathcal{D}(1,2,2)$ distribution, $n$=50]{Boxplots of estimated values of $qZI$ for the $\mathcal{D}(1,2,2)$ distribution, $n$=50}
\label{fig:qZ_20_50}
\end{figure}
\end{minipage}
\hfill
\begin{minipage}[b]{0.45\textwidth}
\begin{figure}[H]
\includegraphics[angle=0,width=\textwidth]
{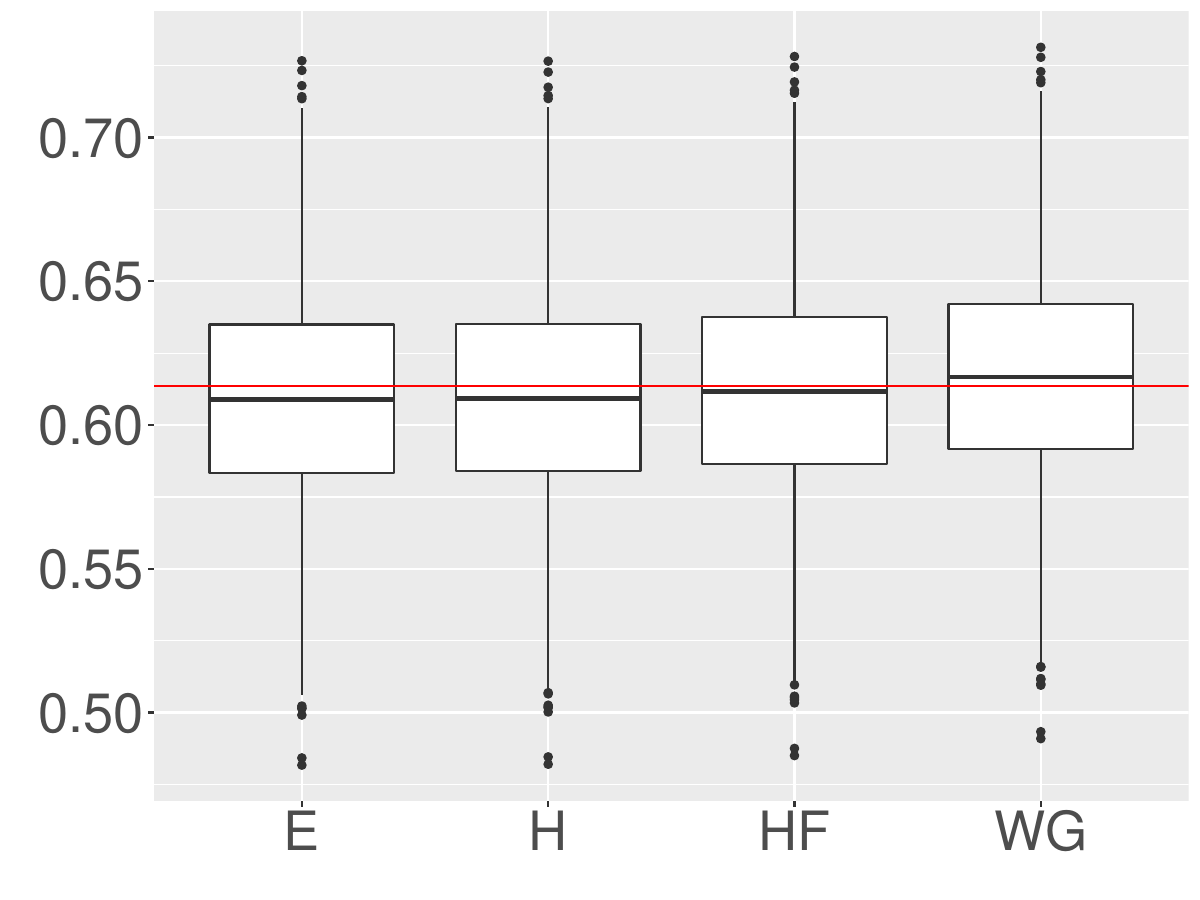}
\caption[Boxplots of estimated values of $qDI$ for the $\mathcal{D}(1,2,2)$ distribution, $n$=50]{Boxplots of estimated values of $qDI$ for the $\mathcal{D}(1,2,2)$ distribution, $n$=50}
\label{fig:qD_20_50}
\end{figure}
\end{minipage}

\section{Simulation study}\label{sec:simulation}

To investigate the performance of the estimators of the proposed measures of inequality
$qZI$ and $qDI$ we conducted a~simulation study.
We performed the simulations using the R programming language.
In the simulation study we compared the accuracy of the plug-in estimators $qZI({\widehat Q}_{n})$ and $qDI({\widehat Q}_{n})$ of the indices $qZI$ and $qDI,$
where ${\widehat Q}_{n}$ is one of the following quantile function estimators:
${\widehat Q}_{n}^{E},$ ${\widehat Q}_{n}^{H},$ ${\widehat Q}_{n}^{WG,}$ and ${\widehat Q}_{n}^{HF},$
obtained by applying the R function {\it quantile} from the package {\it stats} with option {\it type} 1, 5, 6 and 8, respectively.

We generated 1000 repetitions of random samples with sample size $n=50$, $n=100$ and $n=500$
from the Dagum distribution $\mathcal{D}(\sigma,a,b),$ with cdf given by (\ref{e:FDagum}), with various shape parameters $a$ $(0.5,1,2,4)$ and $b$ $(0.5, 1)$.

Dagum distribution was considered here, as it is often used to fit the distribution of incomes (see for example \cite{Bandourian2002}).
For each sample, we calculated the estimated values of the indices $qZI$ and $qDI.$
We compared the medians of the obtained estimates of the indices $qZI$ and $qDI$ with the true values of these indices
and the differences between the upper and lower quartiles (for each quantile function estimator considered).
Moreover, a~similar experiment was performed for Pareto distribution with shape parameters $0.7, 1, 2$ and $3$.

We use mean integrated square error (MISE), defined as
\begin{align*}
    \text{MISE}(\widetilde{qZ}_{n})=\mathbb{E}\left[
\int_{0}^{1}\left[\widetilde{qZ}_{n}(p;{\bf X})-qZ(p;Q)\right]^2dp\right]
\end{align*}
to assess the performance of the estimators $\widetilde{qZ}_{n}$ of $qZ$.
MISE$(\widetilde{qD}_{n})$ is defined analogously.

Figures \ref{fig:qZ_20_50} -- \ref{fig:qD_5_500} contain boxplots of the values of the estimators of $qZI$ and $qDI$
 based on samples of size 50, 100 and 500 from the Dagum distribution $\mathcal{D}(1,2,2),$ and $\mathcal{D}(1,0.5,2)$.
 The red lines represent the true values of the indices in each case, which can also be found in Table \ref{tab:index_values}.
 Tables \ref{tab:mise_50}, \ref{tab:mise_100}, and \ref{tab:mise_500} present the values of MISE of the estimators of $qZ$ and $qD$ with different shape parameters and sample size 50, 100, and 500, respectively.

From the simulations, we have the following conclusions:

\begin{itemize}
    \item Among the four simple estimators, WG is the best (in terms of median unbiasedness and MISE) when the inequality index is high, while HF is the best in the opposite case in terms of median unbiasedness and H or HF are the best in terms of MISE.
    \item The differences between these four estimators are tiny for large $n$ (500 or greater), and any of them can be chosen to be used to analyse the inequality of the observations with no relevant discrepancies.
    \item The MSE and MISE is relatively small when the inequality is very high. This can be explained by the character of the relation between the variance of the empirical estimator of $qZI$ and $qDI$ and their true values, which can be seen in Figures \ref{fig:qZ_var} and \ref{fig:qD_var}, respectively.
    \item Analogous conclusions can be drawn from the analysis of simulation study performed for both Dagum and Pareto distributions, thus the results of the experiment performed for the Pareto distribution are not depicted here, to maintain the conciseness of this section.
\end{itemize}


\begin{minipage}[b]{0.45\textwidth}
\begin{figure}[H]
\includegraphics[angle=0,width=\textwidth]
{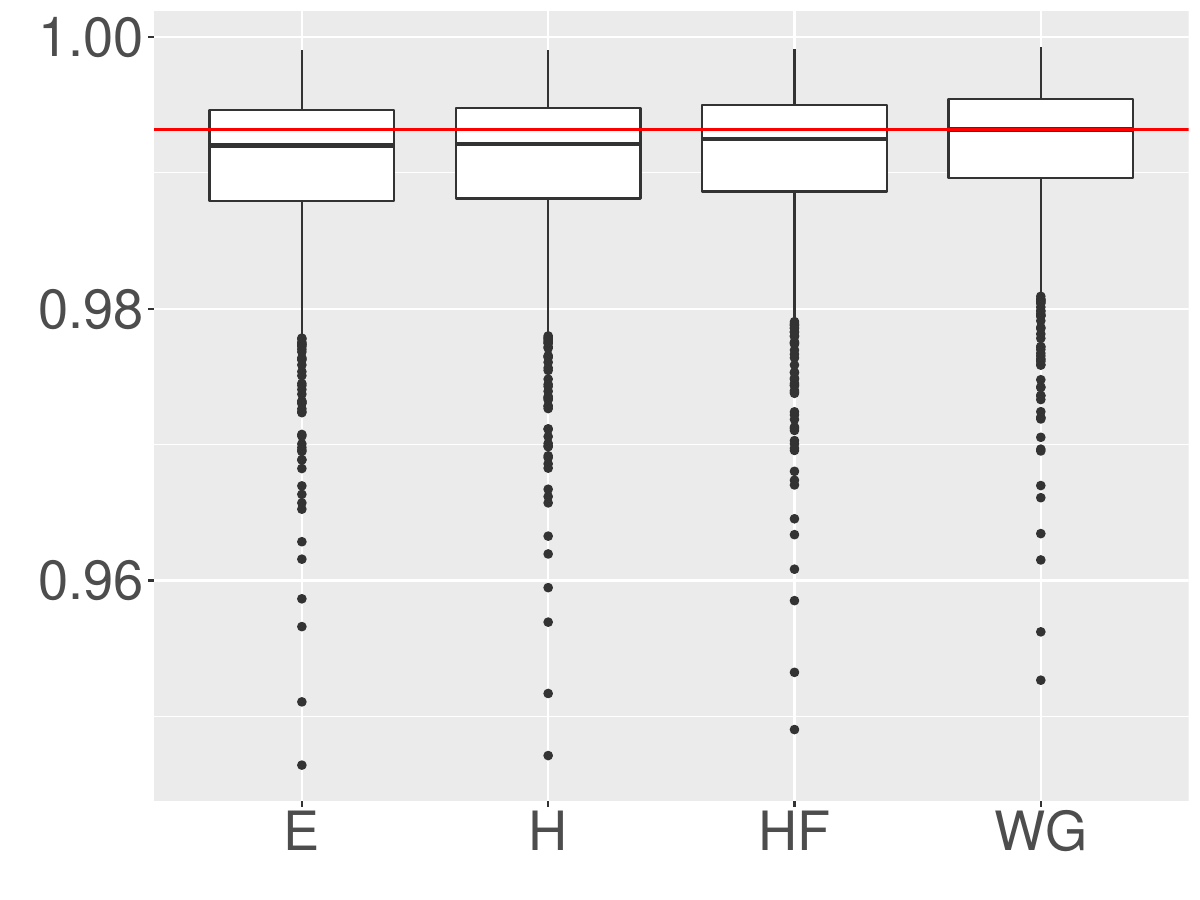}
\caption[Boxplots of estimated values of $qZI$ for the $\mathcal{D}(1,0.5,2)$ distribution, $n$=50]{Boxplots of estimated values of $qZI$ for the $\mathcal{D}(1,0.5,2)$ distribution, $n$=50}
\label{fig:qZ_5_50}
\end{figure}
\end{minipage}
\hfill
\begin{minipage}[b]{0.45\textwidth}
\begin{figure}[H]
\includegraphics[angle=0,width=\textwidth]
{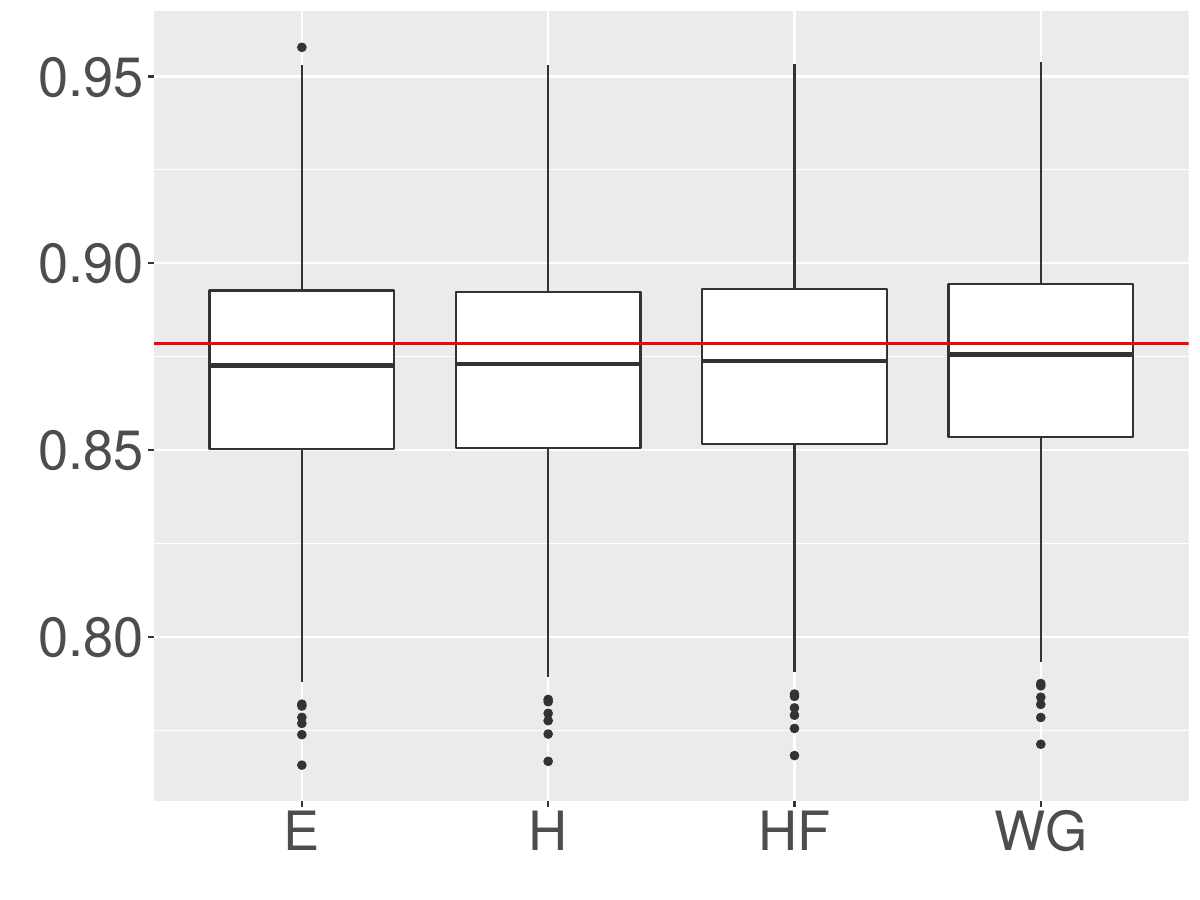}
\caption[Boxplots of estimated values of $qDI$ for the $\mathcal{D}(1,0.5,2)$ distribution, $n$=50]{Boxplots of estimated values of $qDI$ for the $\mathcal{D}(1,0.5,2)$ distribution, $n$=50}
\label{fig:qD_5_50}
\end{figure}
\end{minipage}


\begin{minipage}[b]{0.45\textwidth}
\begin{figure}[H]
\includegraphics[angle=0,width=\textwidth]
{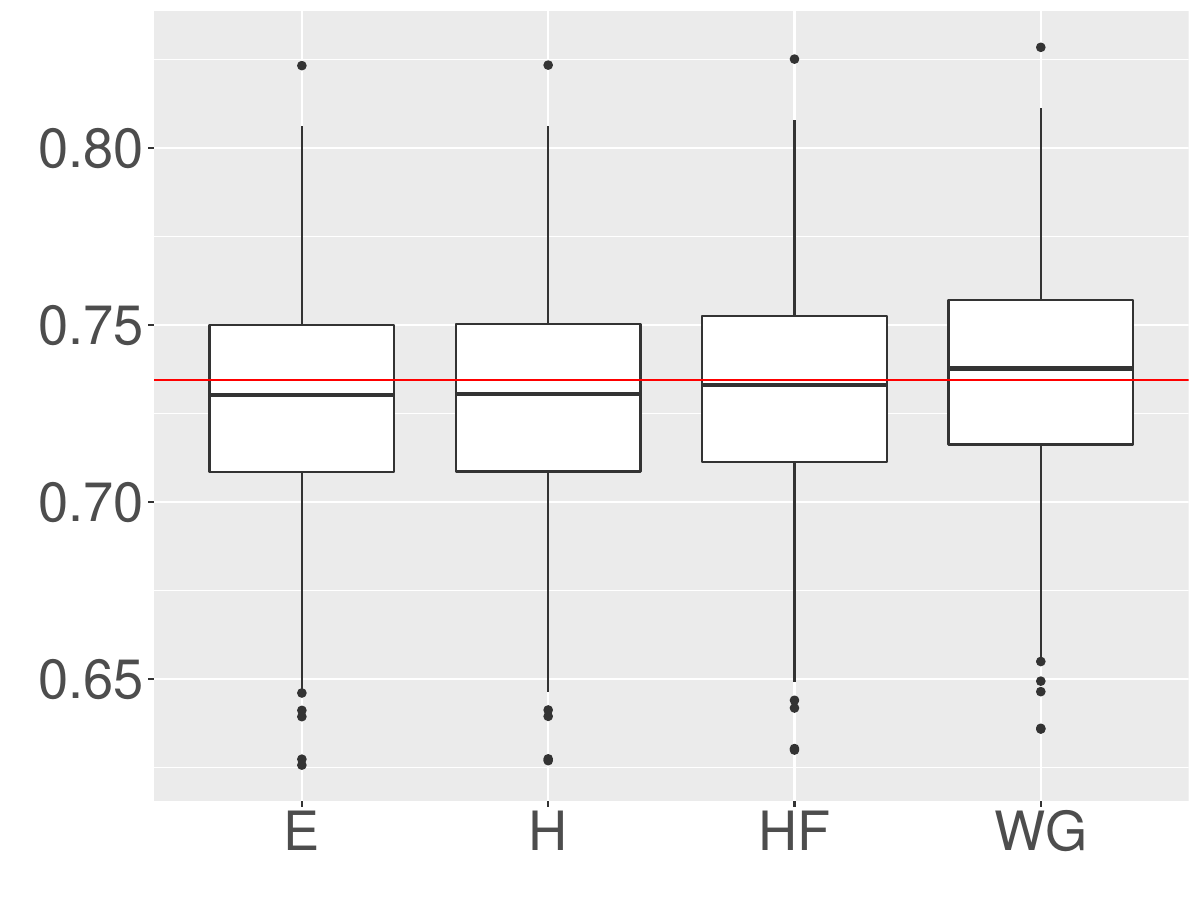}
\caption[Boxplots of estimated values of $qZI$ for the $\mathcal{D}(1,2,2)$ distribution, $n$=100]{Boxplots of estimated values of $qZI$ for the $\mathcal{D}(1,2,2)$ distribution, $n$=100}
\label{fig:qZ_20_100}
\end{figure}
\end{minipage}
\hfill
\begin{minipage}[b]{0.45\textwidth}
\begin{figure}[H]
\includegraphics[angle=0,width=\textwidth]
{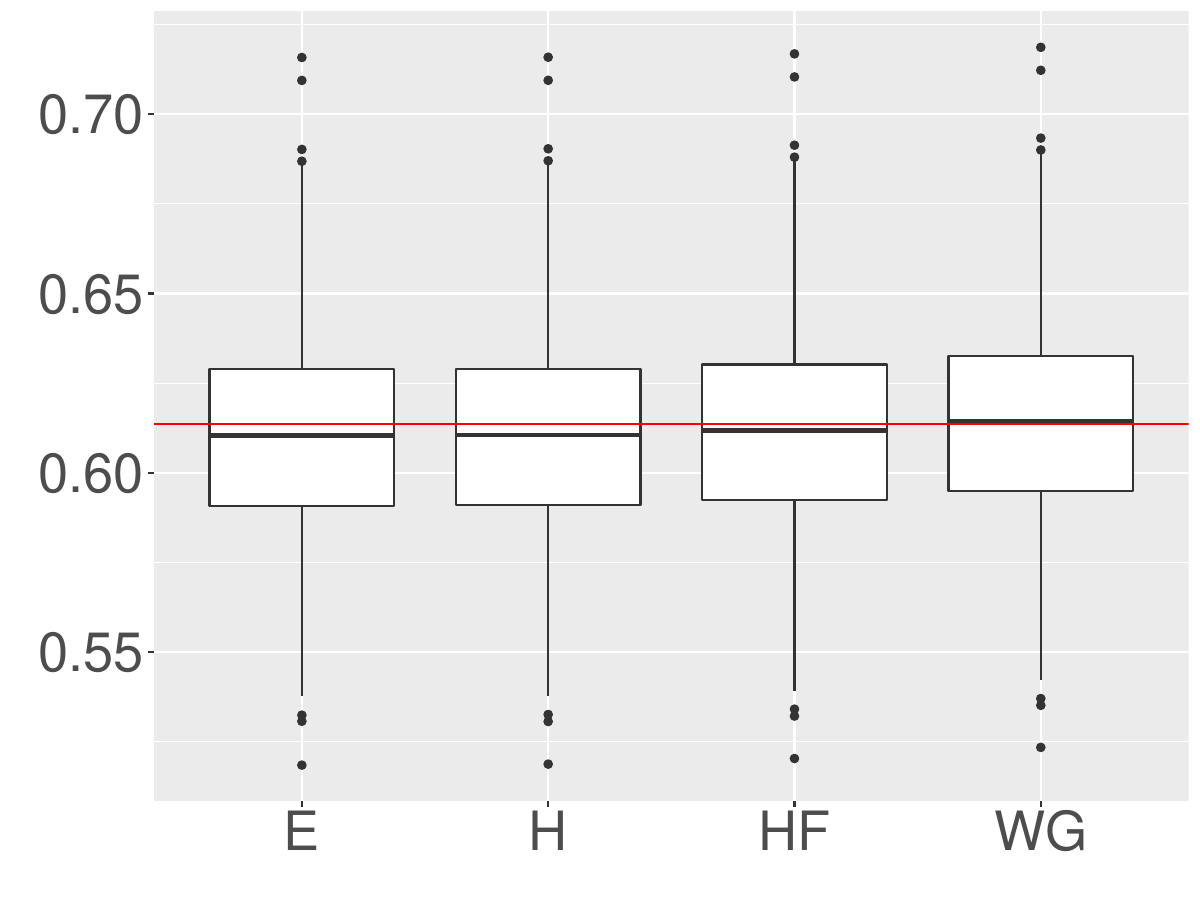}
\caption[Boxplots of estimated values of $qDI$ for the $\mathcal{D}(1,2,2)$ distribution, $n$=100]{Boxplots of estimated values of $qDI$ for the $\mathcal{D}(1,2,2)$ distribution, $n$=100}
\label{fig:qD_20_100}
\end{figure}
\end{minipage}


\begin{minipage}[b]{0.45\textwidth}
\begin{figure}[H]
\includegraphics[angle=0,width=\textwidth]
{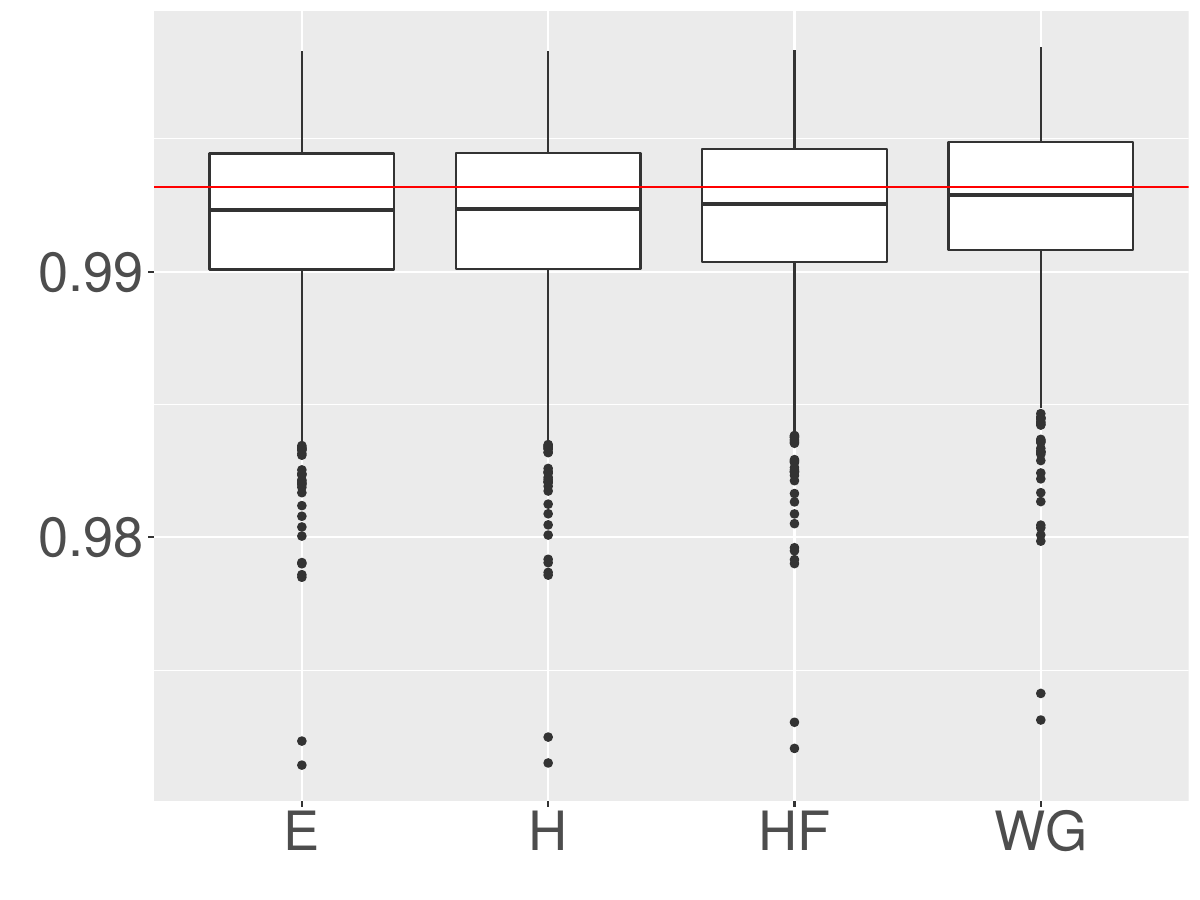}
\caption[Boxplots of estimated values of $qZI$ for the $\mathcal{D}(1,0.5,2)$ distribution, $n$=100]{Boxplots of estimated values of $qZI$ for the $\mathcal{D}(1,0.5,2)$ distribution, $n$=100}
\label{fig:qZ_5_100}
\end{figure}
\end{minipage}
\hfill
\begin{minipage}[b]{0.45\textwidth}
\begin{figure}[H]
\includegraphics[angle=0,width=\textwidth]
{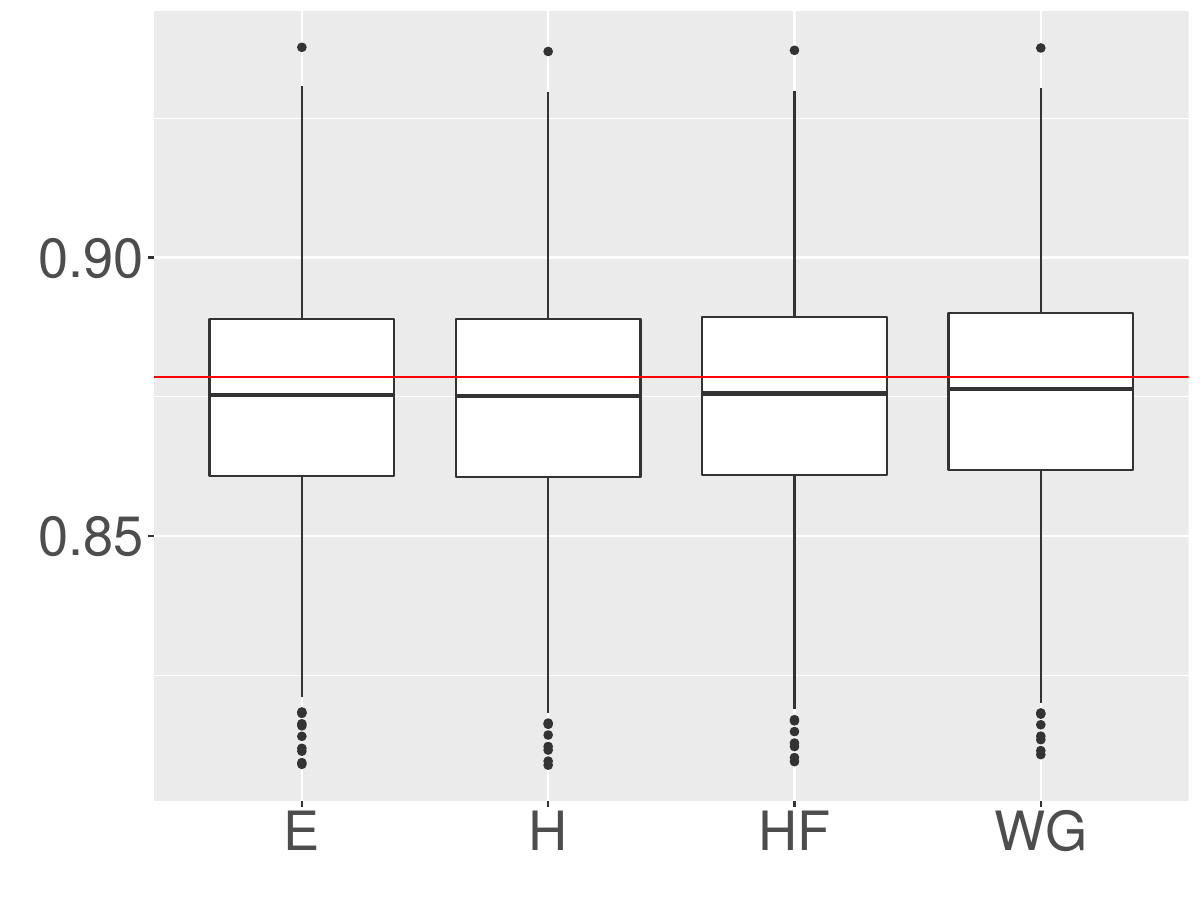}
\caption[Boxplots of estimated values of $qDI$ for the $\mathcal{D}(1,0.5,2)$ distribution, $n$=100]{Boxplots of estimated values of $qDI$ for the $\mathcal{D}(1,0.5,2)$ distribution, $n$=100}
\label{fig:qD_5_100}
\end{figure}
\end{minipage}


\begin{minipage}[b]{0.45\textwidth}
\begin{figure}[H]
\includegraphics[angle=0,width=\textwidth]
{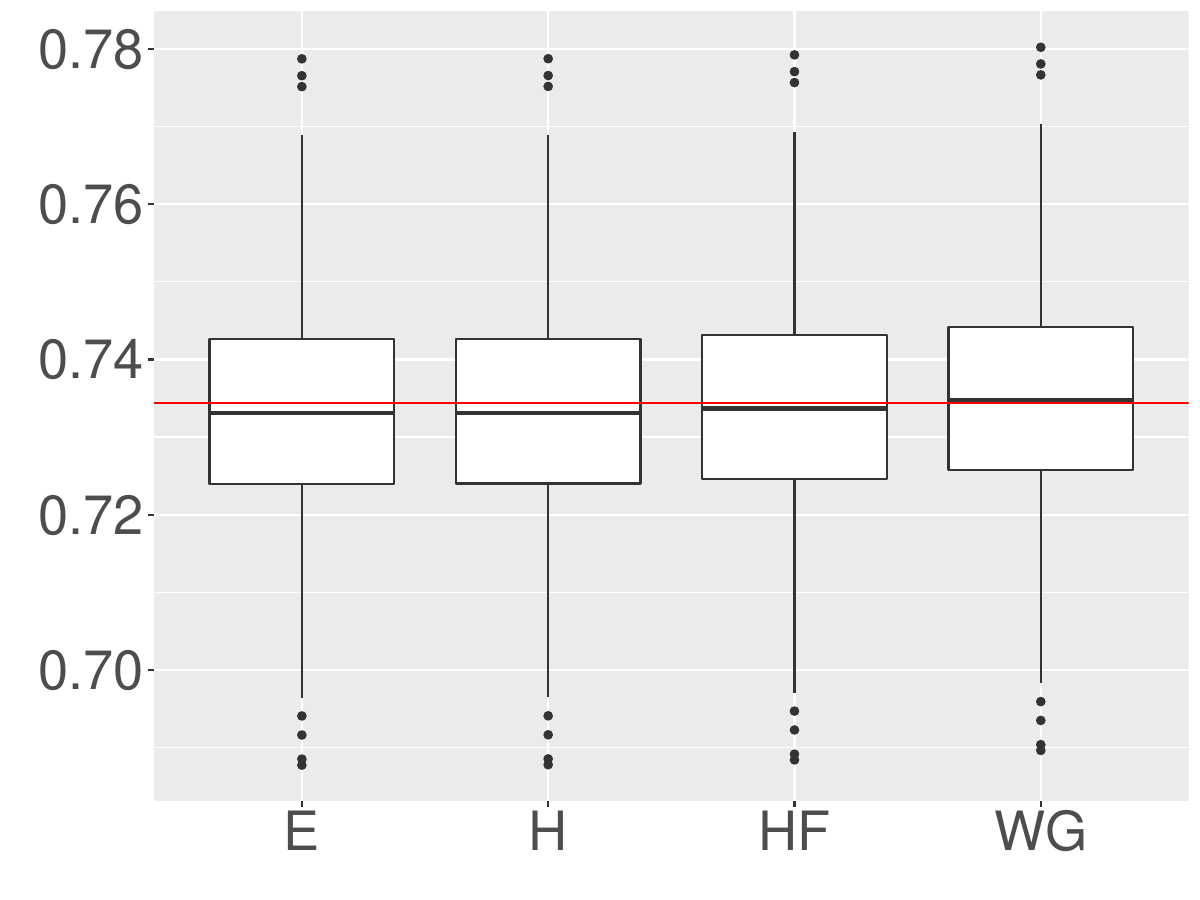}
\caption[Boxplots of estimated values of $qZI$ for the $\mathcal{D}(1,2,2)$ distribution, $n$=500]{Boxplots of estimated values of $qZI$ for the $\mathcal{D}(1,2,2)$ distribution, $n$=500}
\label{fig:qZ_20_500}
\end{figure}
\end{minipage}
\hfill
\begin{minipage}[b]{0.45\textwidth}
\begin{figure}[H]
\includegraphics[angle=0,width=\textwidth]
{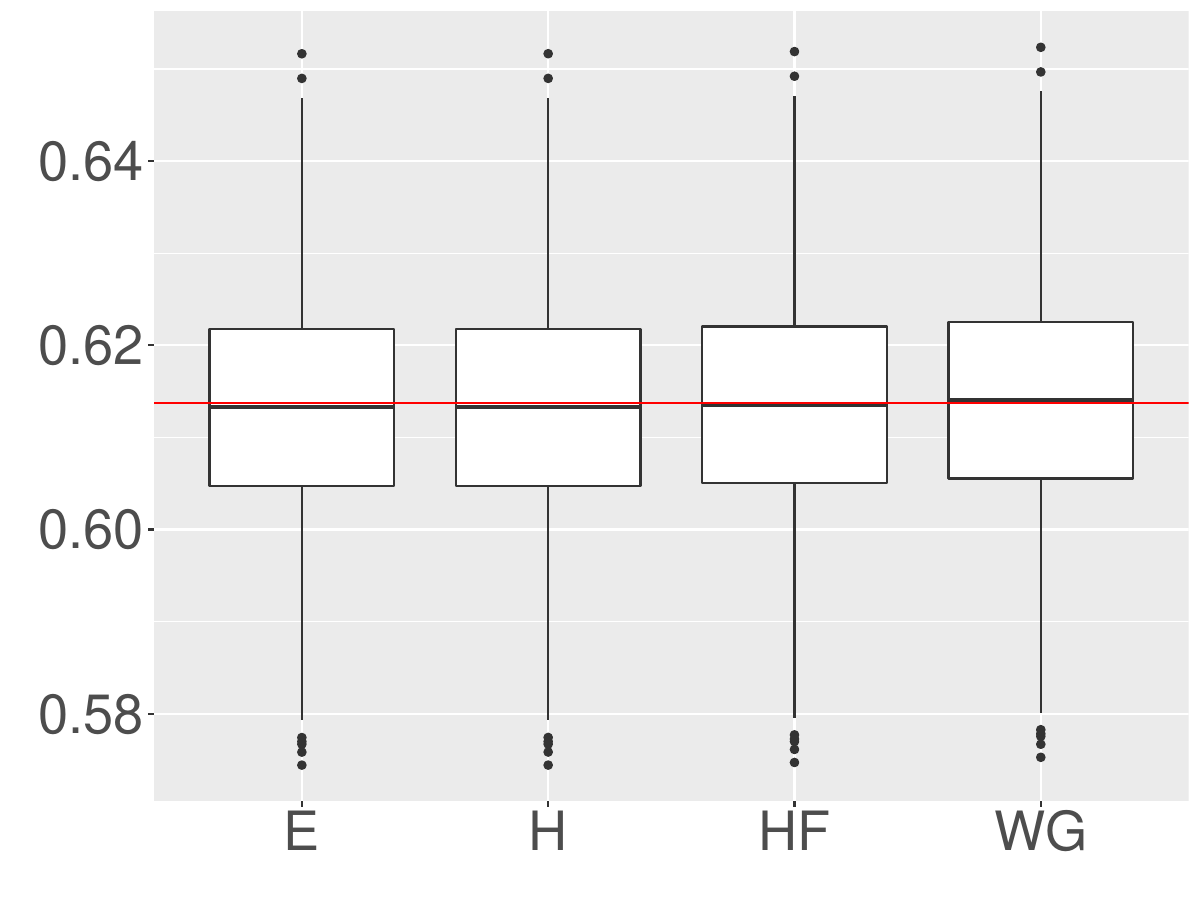}
\caption[Boxplots of estimated values of $qDI$ for the $\mathcal{D}(1,2,2)$ distribution, $n$=500]{Boxplots of estimated values of $qDI$ for the $\mathcal{D}(1,2,2)$ distribution, $n$=500}
\label{fig:qD_20_500}
\end{figure}
\end{minipage}


\begin{minipage}[b]{0.45\textwidth}
\begin{figure}[H]
\includegraphics[angle=0,width=\textwidth]
{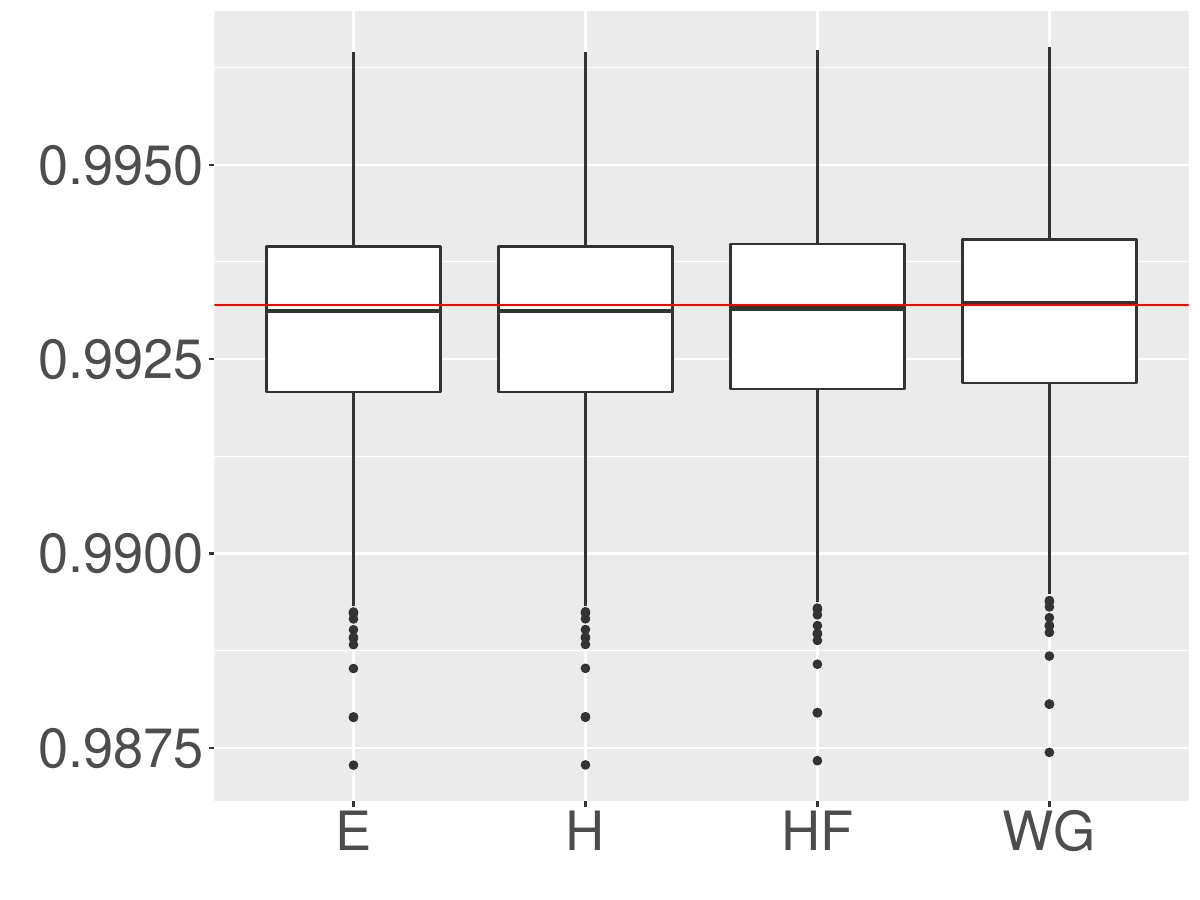}
\caption[Boxplots of estimated values of $qZI$ for the $\mathcal{D}(1,0.5,2)$ distribution, $n$=500]{Boxplots of estimated values of $qZI$ for the $\mathcal{D}(1,0.5,2)$ distribution, $n$=500}
\label{fig:qZ_5_500}
\end{figure}
\end{minipage}
\hfill
\begin{minipage}[b]{0.45\textwidth}
\begin{figure}[H]
\includegraphics[angle=0,width=\textwidth]
{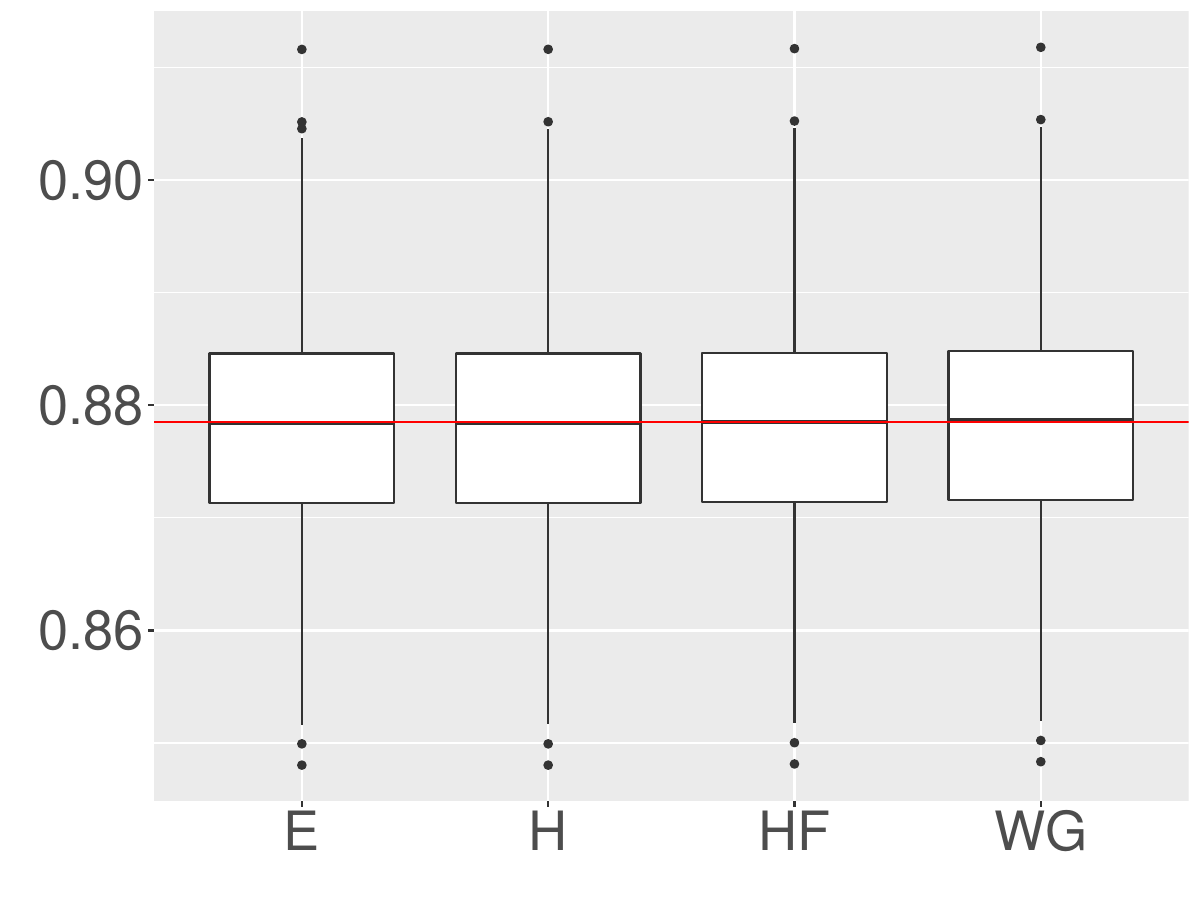}
\caption[Boxplots of estimated values of $qDI$ for the $\mathcal{D}(1,0.5,2)$ distribution, $n$=500]{Boxplots of estimated values of $qDI$ for the $\mathcal{D}(1,0.5,2)$ distribution, $n$=500}
\label{fig:qD_5_500}
\end{figure}
\end{minipage}

\begin{table}[h]
\centering
\caption{MISE (multiplied by 1 000) of the estimators of $qZ$ and $qD$ for sample size $50$ and various shape parameters $b$ and $a$}
\label{tab:mise_50}
\begin{tabular}{c|c|cccc|cccc}
\hline
\multicolumn{2}{c|}{ } & \multicolumn{4}{c|}{$qZ$} & \multicolumn{4}{c}{$qD$} \\
\hline
$b$ & $a$ & E & H & WG & HF & E & H & WG & HF \\
\hline
\multirow{4}{*}{.5} & .5 & 0.0137 & 0.0126 & \textbf{0.0092} & 0.0113 & 4.7540 & 3.5226 & \textbf{3.3925} & 3.4768 \\
 & .8 & 0.2874 & 0.2734 & \textbf{0.2181} & 0.2526 & 5.0982 & 4.2547 & \textbf{4.1104} & 4.2021 \\
 & 2 & 3.4037 & 3.2607 & \textbf{3.0349} & 3.1544 & 3.8707 & 3.4990 & \textbf{3.3915} & 3.4526 \\
 & 4 & 5.0844 & 4.8493 & 5.0229 & \textbf{4.8468} & 3.0898 & 2.8419 & 2.8490 & \textbf{2.8212} \\
\hline
\multirow{4}{*}{1} & .5 & 0.0995 & 0.0935 & \textbf{0.0722} & 0.0855 & 4.9898 & 4.0245 & \textbf{3.8825} & 3.9741 \\
 & .8 & 0.7945 & 0.7554 & \textbf{0.6289} & 0.7069 & 4.4662 & 3.8259 & \textbf{3.6944} & 3.7775 \\
 & 2 & 4.3915 & 4.2032 & \textbf{4.0838} & 4.1153 & 3.3947 & 3.0950 & \textbf{3.0394} & 3.0624 \\
 & 4 & 4.7102 & \textbf{4.4960} & 4.8447 & 4.5544 & 2.5593 & 2.3536 & 2.4342 & \textbf{2.3511} \\
\hline
\end{tabular}
\end{table}

\begin{table}[h]
\centering
\caption{MISE (multiplied by 1 000) of the estimators of $qZ$ and $qD$ for sample size $100$ and various shape parameters $b$ and $a$}
\label{tab:mise_100}
\begin{tabular}{c|c|cccc|cccc}
\hline
\multicolumn{2}{c|}{ } & \multicolumn{4}{c|}{$qZ$} & \multicolumn{4}{c}{$qD$} \\
\hline
$b$ & $a$ & E & H & WG & HF & E & H & WG & HF \\
\hline
\multirow{4}{*}{.5} & .5 & 0.0037 & 0.0036 & \textbf{0.0030} & 0.0034 & 2.3810 & 2.0448 & \textbf{2.0086} & 2.0325  \\
 & .8 & 0.1075 & 0.1053 & \textbf{0.0929} & 0.1008 & 2.3604 & 2.1467 & \textbf{2.1106} & 2.1339 \\
 & 2 & 1.6423 & 1.5999 & \textbf{1.5340} & 1.5685 & 1.8998 & 1.8050 & \textbf{1.7809} & 1.7940 \\
 & 4 & 2.6824 & 2.5965 & 2.6616 & \textbf{2.5963} & 1.5709 & 1.5071 & 1.5029 & \textbf{1.4987} \\
\hline
\multirow{4}{*}{1} & .5 & 0.0311 & 0.0304 & \textbf{0.0261} & 0.0288 & 2.3543 & 2.1036 & \textbf{2.0635} & 2.0897 \\
 & .8 & 0.3923 & 0.3838 & \textbf{0.3470} & 0.3702 & 2.2467 & 2.0830 & \textbf{2.0434} & 2.0693 \\
 & 2 & 2.2329 & 2.1674 & \textbf{2.1365} & 2.1418 & 1.7845 & 1.7078 & \textbf{1.6818} & 1.6952 \\
 & 4 & 2.3518 & \textbf{2.2630} & 2.3949 & 2.2827 & 1.2421 & 1.1834 & 1.1979 & \textbf{1.1789} \\
\hline
\end{tabular}
\end{table}

\begin{table}[h]
\centering
\caption{MISE (multiplied by 1 000) of the estimators of $qZ$ and $qD$ for sample size $500$ and various shape parameters $b$ and $a$}
\label{tab:mise_500}
\begin{tabular}{c|c|cccc|cccc}
\hline
\multicolumn{2}{c|}{ } & \multicolumn{4}{c|}{$qZ$} & \multicolumn{4}{c}{$qD$} \\
\hline
$b$ & $a$ & E & H & WG & HF & E & H & WG & HF \\
\hline
\multirow{4}{*}{.5} & .5 & \textbf{0.0005} & \textbf{0.0005} & \textbf{0.0005} & \textbf{0.0005} & 0.4748 & 0.4615 & \textbf{0.4607} & 0.4616 \\
 & .8  & 0.0159 & 0.0159 & \textbf{0.0154} & 0.0157 & 0.4259 & 0.4180 & \textbf{0.4170} & 0.4177 \\
 & 2 & 0.3077 & 0.3055 & \textbf{0.3039} & 0.3044 & 0.3617 & 0.3580 & \textbf{0.3576} & 0.3577 \\
 & 4 &  0.5376 & \textbf{0.5311} & 0.5347 & 0.5314 & 0.3090 & 0.3067 & \textbf{0.3044} & 0.3057 \\
\hline
\multirow{4}{*}{1} & .5 & 0.0045 & 0.0045 & \textbf{0.0043} & 0.0044 & 0.4318 & 0.4227 & \textbf{0.4213} & 0.4219 \\
 & .8 & 0.0634 & 0.0631 & \textbf{0.0613} & 0.0625 & 0.4228 & 0.4178 & \textbf{0.4160} & 0.4171 \\
 & 2 & 0.4289 & 0.4251 & \textbf{0.4232} & 0.4235 & 0.3353 & 0.3327 & \textbf{0.3317} & 0.3322 \\
 & 4 & 0.4934 & \textbf{0.4865} & 0.5046 & 0.4898 & 0.2492 & \textbf{0.2467} & 0.2483 & 0.2468 \\
\hline
\end{tabular}
\end{table}

According to the results of our simulations, we recommend using in practice estimators based on the quantile function estimator $\widehat{Q}_{n}^{WG}$ when the value of  the $\widehat{qZI}$
(or $\widehat{qDI}$) is high and  $\widehat{Q}_{n}^{HF}$ when it is small (0.6 and less).

\begin{remark}
    The computation of $E,~H$, $WG$ and $HF$ estimators of $qZI$ and $qDI$ is significantly faster when using the closed-form expression, than using \textit{integrate} function in R. For sample size $n=1,000$ it is approximately 100 times and for $n=10,000$ approximately 10 times faster.
\end{remark}

\section{Real data analysis}\label{sec:real}
Two examples of data with incomes are analysed in this section. In the first case we compare the level of inequality between three groups of employees divided according to their seniority, while in the second one the inequality of lower bands and upper bands of salaries from job offers is compared.
\subsection{Academic salaries}
As an example, we consider the data set {\it Salaries} from {\it carData} package~in~R.~The set contains, among others,
the 2008-09 nine-month academic salary for Assistant Professors ($n=68$), Associate Professors ($n=64$) and Professors ($n=266$) in a~college in the US.
Table \ref{tab:salary} contains the values of estimators $\widehat{qZI}_{n}^{HF}$ and $\widehat{qDI}_{n}^{HF}$ of the indices $qZI$ and $qDI$
in the subgroups (Assistant Professors, Associate Professors and Professors) and in the entire study group.
The corresponding estimators $\widehat{qZ}_{n}^{HF}$ and $\widehat{qD}_{n}^{HF}$ of the concentration curves drawn for each group are depicted in figures \ref{fig:qZ_academic_salaries} and \ref{fig:qD_academic_salaries}, respectively.

We used estimators based on $\widehat{Q}_n^{HF}$ because according to the results of our simulations we recommend them when the income distribution does not have outliers or heavy tails.

\begin{table}[H]
	\caption{Estimates of $qZI$ and $qDI$ for salary distribution}
	\label{tab:salary}
	\centering
\begin{tabular}[H]{c|c|c|c|c}
 & Prof & AssocProf & AsstProf & All \\ \hline
 $qZI$ & 0.2973 & 0.2225 & 0.1580 & 0.3453 \\ \hline
 $qDI$ & 0.2774 & 0.2138 & 0.1535 & 0.3185
\end{tabular}
\end{table}


\begin{minipage}[b]{0.45\textwidth}
\begin{figure}[H]
\includegraphics[angle=0,width=\textwidth]
{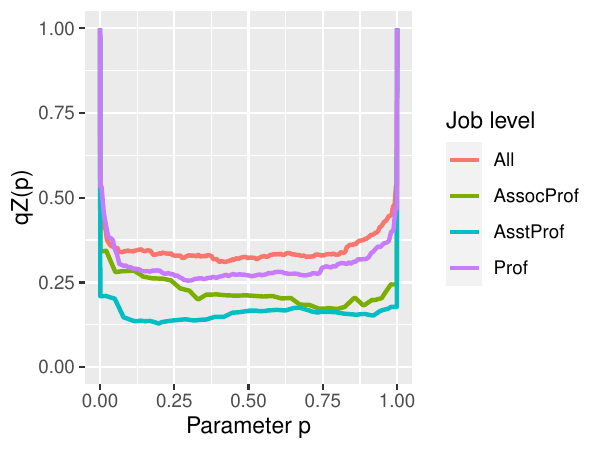}
\caption[HF Estimators of $qZ$ curves for salaries of academics]{HF Estimators of $qZ$ curves for salaries of academics}
\label{fig:qZ_academic_salaries}
\end{figure}
\end{minipage}
\hfill
\begin{minipage}[b]{0.45\textwidth}
\begin{figure}[H]
\includegraphics[angle=0,width=\textwidth]
{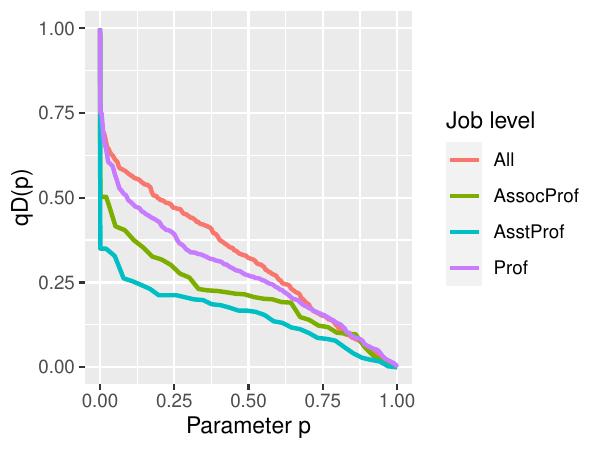}
\caption[HF Estimators of $qD$ curves for salaries of academics]{HF Estimators of $qD$ curves for salaries of academics}
\label{fig:qD_academic_salaries}
\end{figure}
\end{minipage}

As expected, the highest estimated values of indices $qZI,$ $qDI$ are in the entire study group,
and the lowest -- in the subgroup of Assistant Professors.

\subsection{Job offers for data scientists}
A~data set with job offers for data scientists was prepared by scraping data from a~website with job offers \cite{JobOffers} in the US. It consists of 742 records with 42 variables, including the lower and upper bands of the yearly salary (in USD) announced in each offer. We can measure the inequality of the distribution of the lower and upper bands and check which of them has higher inequality. Similarly as in the previous case, $\widehat{qZI}_{n}^{HF}$ and $\widehat{qDI}_{n}^{HF}$ are used to estimate the values of the indices $qZI$ and $qDI$. As we can see in Table \ref{tab:bands}, the inequality is higher among the lower bands than among the upper bands. The corresponding estimators $\widehat{qZ}_{n}^{HF}$ and $\widehat{qD}_{n}^{HF}$ of the concentration curves are depicted in Figures \ref{fig:qZ_job_offers} and \ref{fig:qD_job_offers}, respectively.


\begin{minipage}[b]{0.45\textwidth}
\begin{figure}[H]
\includegraphics[angle=0,width=\textwidth]
{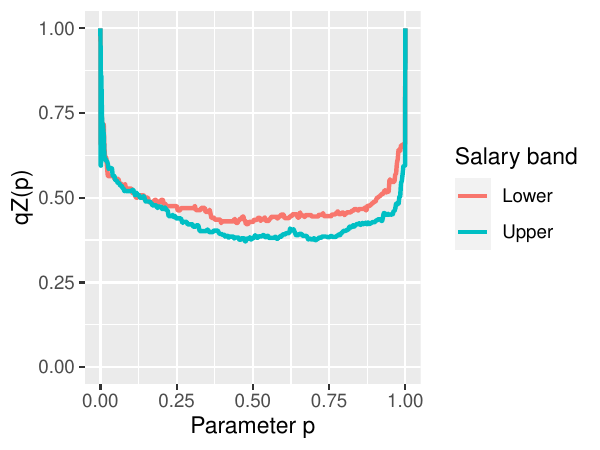}
\caption[HF Estimators of $qZ$ curves for lower and upper bands of salary offers]{HF Estimators of $qZ$ curves for lower and upper bands of salary offers}
\label{fig:qZ_job_offers}
\end{figure}
\end{minipage}
\hfill
\begin{minipage}[b]{0.45\textwidth}
\begin{figure}[H]
\includegraphics[angle=0,width=\textwidth]
{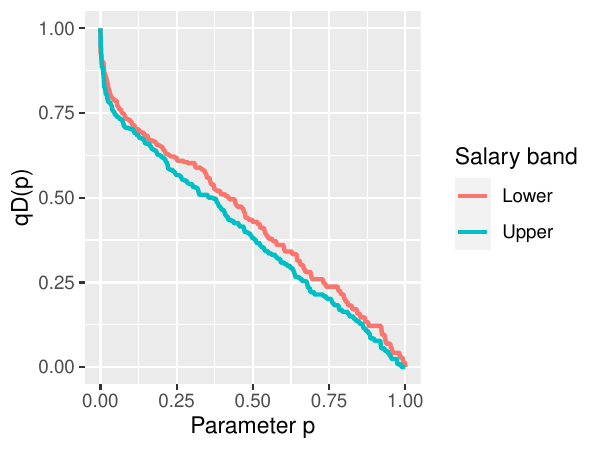}
\caption[HF Estimators of $qD$ curves for lower and upper bands of salary offers]{HF Estimators of $qD$ curves for lower and upper bands of salary offers}
\label{fig:qD_job_offers}
\end{figure}
\end{minipage}

\begin{table}[H]
	\caption{Estimates of $qZI$ and $qDI$ for lower and upper bands of salary in job offers}
	\label{tab:bands}
	\centering
\begin{tabular}[H]{c|c|c}
 & Lower band & Upper band  \\ \hline
 $qZI$ & 0.4793 & 0.4351  \\ \hline
 $qDI$ & 0.4295 & 0.3891
\end{tabular}
\end{table}

\section{Concluding remarks and some prospects}\label{sec:concluding}
We described the quantile versions of known inequality curves and inequality measures.
We have especially considered the quantile versions of the Zenga-07 and $D$ curves and indices.
We also proposed the nonparametric estimators of the curves and indices.
The fact that the Zenga-07 and $D$ curves and inequality measures are defined only for distributions with finite expected value was our motivation to consider their quantile versions.
An additional argument for using the quantile versions of the inequality curves and inequality measures, even in cases of distributions with finite expected value,
can be the ``robustness'' of their estimators, understood as resistance to outliers.

Investigating properties, applications, and estimators of other inequality measures appearing in the literature might be a~good topic for future research.
Another interesting issue might be the parametric estimation of the aforementioned curves and measures.
We are going to study these issues in future work.

\bibliographystyle{acm}

\end{document}